\documentclass[11pt,reqno]{amsart}

\usepackage{amsmath}
\usepackage{amsthm}
\usepackage{amssymb}

\theoremstyle{plain}
\newtheorem{theorem}{Theorem}
\newtheorem{defin}{Definition}
\newtheorem{prop}{Proposition}
\newtheorem{lem}{Lemma}

\theoremstyle{remark}
\newtheorem{remark}{Remark}

\newcommand{\beq}{\begin{equation}}
\newcommand{\eeq}{\end{equation}}
\newcommand{\beql}[1]{\begin{equation}\label{#1}}
\newcommand{\beal}[1]{\begin{equation}\label{#1}\begin{aligned}}
\newcommand{\eeal}{\end{aligned}\end{equation}}

\newcommand{\ben}{\begin{eqnarray}}
\newcommand{\een}{\end{eqnarray}}

\DeclareMathOperator{\sgn}{sgn}
\DeclareMathOperator{\dive}{div}
\DeclareMathOperator{\vol}{vol}
\renewcommand{\b}[1]{\ensuremath{\mathbf{#1}}} % bold type for vectors

\usepackage{color}
\usepackage[left=3cm,right=3cm,top=3.5cm,bottom=4cm,marginparwidth=1.7in]{geometry}
\usepackage{graphicx}
\usepackage{subfigure}
\usepackage[toc,page]{appendix}

\def\PicPath{}

\begin{document}

\title[Towards monitoring microscopic parameters for electropermeabilization]{Towards monitoring critical  microscopic parameters for electropermeabilization}

\author{H. Ammari}
\address{Department of Mathematics, 
ETH Z\"urich, 
R\"amistrasse 101, CH-8092 Z\"urich, Switzerland}
\email{habib.ammari@math.ethz.ch}
\thanks{This work was supported  by the ERC Advanced Grant Project MULTIMOD--267184.}

\author{T. Widlak}
\address{Department of Mathematics and Applications,
Ecole Normale Sup\'erieure, 45, rue d'Ulm, 75230 Paris Cedex 05, France}
\email{thomas.widlak@ens.fr}

\author{W. Zhang}
\address{Department of Mathematics and Applications,
Ecole Normale Sup\'erieure, 45, rue d'Ulm, 75230 Paris Cedex 05, France}
\email{wenlong.zhang@ens.fr}

% Old version:
%\title{Towards monitoring critical  microscopic parameters for electropermeabilization\thanks{\footnotesize This work was supported  by the ERC Advanced Grant Project MULTIMOD--267184.}}
%
%\author{H. Ammari\thanks{\footnotesize Department of Mathematics, 
%ETH Z\"urich, 
%R\"amistrasse 101, CH-8092 Z\"urich, Switzerland (habib.ammari@math.ethz.ch). } \and 
%T. Widlak\thanks{\footnotesize Department of Mathematics and Applications,
%Ecole Normale Sup\'erieure, 45 Rue d'Ulm, 75005 Paris, France
%(thomas.widlak@ens.fr, wenlong.zhang@ens.fr).} \and W. Zhang\footnotemark[3]}
%
%\date{}

\begin{abstract}
Electropermeabilization is a clinical technique in cancer treatment to locally
stimulate the cell metabolism. It is based on electrical fields that change the
properties of the cell membrane. With that, cancer treatment can reach the cell more
easily. Electropermeabilization occurs only with accurate dosage of the electrical field. For
applications, a monitoring for the amount of electropermeabilization is needed.
It is a first step to image the macroscopic electrical field during the process. Nevertheless, this is not complete, because electropermeabilization depends on critical individual properties of the cells
such as their curvature. From the macroscopic field, one cannot directly infer that  microscopic state.
In this article, we study effective parameters in a homogenization model as the next
step to monitor the microscopic properties in clinical practice. We start from a physiological
cell model for electropermeabilization and analyze its well-posedness. For a
dynamical homogenization scheme, we prove convergence and then analyze the effective
parameters, which can be found by macroscopic imaging methods. We demonstrate
numerically the sensitivity of these effective parameters to critical microscopic parameters governing electropermeabilization. This opens the door to solve the inverse problem of rreconstructing these parameters. \end{abstract}

\maketitle

\bigskip

\noindent {\footnotesize Mathematics Subject Classification
(MSC2000): 35B30, 35R30.}


\section{Introduction}
The technique of electropermeabilization (formerly referred to as electroporation) is employed to make the
chemotherapeutical treatment of cancer more efficient and avoid side-effects.
Instead of spreading out drugs over the whole body, electropermeabilization makes it possible
to focus drug application on special areas. The mechanism of electropermeabilization relies
on careful exposition of biological tissue to electrical fields: this changes the
membrane properties of the cells such that treatment can enter more easily just
at precisely defined areas of the tissue \cite{Ivo10,MikSnoZupKosCer10}.

The local change in microscopic tissue properties, which electropermeabilization effects, occurs only with field strengths above a certain threshold. On
the other hand, too strong fields result in cell death. One therefore thinks of
electropermeabilization occurring within a certain threshold of intensity of the local
electric field \cite{DerMik14}.

For treatment planning in electropermeabilization, one is interested in the percentage of
electroporated cells over the whole tissue to form decisions in the short term how
to gear treatment \cite{KraMarBajCemSer15,DerMik14,MikSnoZupKosCer10}.

One would like supervise the electropermeabilization using measurements of the electric field
distribution with image modalities like in \cite{KraMarBajCemSer15}. In that work, measurements of magnetic resonance electrical impedance tomography  \cite{mreit} have been employed to find the electrical field distribution. A threshold is then applied to find the electroporated cells.

Yet this approach is only the first step in a larger program:
\begin{itemize}
\item the electrical field distribution reconstructed by an imaging modality is a
macroscopic quantity;
\item the thresholding hypothesis is a simplification and should be refined
\cite{DerMik14};
\item the minimum transmembrane voltage governing electropermeabilization is determined by
specific cell characteristics like the curvature of the cell membrane 
\cite{TowKotPucFirMoz08}.
\end{itemize}
One solution to find about microscopic parameters from measurements is to take general models and do a specific parameter fitting with preselected cells like in
\cite{DerMik14}. In clinical practice, though, a preselected cell population may be
unavailable for the analysis.

In this paper, we tackle the next step in electropermeabilization  monitoring and investigate the question to determine microscopic parameters from macroscopic measurements.  The modelling used stems from general physiological tissue models for cells,
asymptotically simplified by Neu and Krassowska~\cite{NeuKra99}. Whereas the
mathematical well-posedness of the model of that model is not available in the
literature, there exists an investigation of well-posedness for a similar model in \cite{KavLegPoiWey14}. In this paper, demonstrate the local
well-posedness of the asymptotic cell model of \cite{NeuKra99}, as well as the absence of a blow up. A  variant of the model is shown to be globally well-posed.

In order to describe the relation between macroscopic and microscopic quantities, we apply the homogenization scheme in \cite{AmaAndBisGia04} to the cell model of Neu and Krassowska \cite{NeuKra99}. This not only describes isotropic effective parameters such as classical theory \cite{ PavSliMik02}, but includes also anisotropy. We provide a
convergence analysis for the homogenized solution.

Then we study numerically the sensitivity of the effective parameters to:
\begin{itemize}
\item the conductivities of the extra- and intracellular media;
\item the shape of the cell membrane;
\item the volume fraction of the cells; 
\item the lattice structure of the cells.
\end{itemize}
We refer to research in \cite{
TowKotPucFirMoz08,IvoVilMir10,MikBerSemCemDem98, PavPavMik02,PucKotValMik06}, where these critical parameters for electropermeabilization have been investigated, partly from an empirical or computer simulation point of view.

The structure of the paper is as follows. In Section \ref{sec:Model}, we introduce
the model of \cite{NeuKra99} on the cellular scale. In
Section \ref{sec:Wellposed}, we investigate its well-posedness properties. In Section \ref{sec:Homogenization}, we perform the homogenization
and show the convergence of the homogenized solution. In  
Section \ref{sec:Illustration}, we provide a sensitivity analysis of the effective
parameters, showing dependence on microscopic properties, summarized in Table~\ref{tab:Sensitivity}. A discussion and final remarks in Section~
\ref{sec:Discussion} conclude the article.

\section{Modelling electropermeabilization on the cellular scale}
\label{sec:Model}
\subsection{Membrane model}
Let $Y\Subset\mathbb{R}^d$ be a bounded domain representing the cell, and let $\Gamma \subset Y$ be the membrane of the cell. Let 
	\[
Y \setminus \Gamma =Y_i\cup Y_e,
\]
where $Y_i$ (resp. $Y_e$) is be the inner (resp. the outer) domain. Let $\sigma_i(x)$ be the conductivity of the cell domain $Y_i$, and $\sigma_e(x)$ be the conductivity outside the cells on $Y_e$.

Let $u_0$ be an imposed voltage on the boundary of $Y$. An electrostatic model for the electrical field $u(x,t)$ on $Y$ in the inner and outer domain is
\begin{align}
\label{eq:Divergence}
	\nabla \cdot (\sigma(x)\;\nabla u(x,t)) &= 0  &&\text{on } Y\setminus\Gamma=Y_i\cup Y_e,\\
\label{eq:BoundaryFar}
u &= u_0 &&\text{on } \partial Y   \text{, with }\Delta u_0=0 \text{ in } Y,\\
\label{eq:BoundaryEq}
	\sigma_e\b n\cdot\nabla u^+  = \sigma_i\b n\cdot\nabla u^-=\sigma  \b n\cdot\nabla u &= \sigma \partial_n u&&\text{on }\Gamma.
\end{align}
Here and throughout this paper, $\partial_n$ denotes the normal derivative. 
 
\subsection{Electropermeabilization models}

In addition to the membrane model, a time-varying conductivity $ \sigma_m(x,t)$ for $x\in\Gamma$ is taken account of.  The general effect of electropermeabilization  is described by relating $\sigma_m$ and the membrane thickness $\delta$ to the transmembrane potential (TMP) jump $u^+(x,t)-u^-(x,t):=[u](x,t)$ in an ordinary differential equation (ODE) on $\Gamma$:
\begin{align}
\label{eq:BoundaryVal}
	\sigma(x)\b n\cdot\nabla u(x,t) &= \frac{c_m}{\delta}\partial_t [u](x,t) + \frac{\sigma_m\left( [u](x,t),t\right) }{\delta}[u](x,t) &\text{ on }\Gamma.
\end{align}
	Here, the vector $\b n$ is the outward normal to $\Gamma$, $\partial_n$ is the normal derivative, the superscripts $\pm$ denote the limits for outside and inside $Y_i$, and $c_m$ is a positive constant.

The membrane conductivity $\sigma_m$ in \eqref{eq:BoundaryVal} is described by different models. In \cite{IvoVilMir10}, Mir et al. propose a static model based on
\begin{equation}
\label{eq:IvorraMir}
\sigma_m([u]) = \sigma_{m0} + K\;(e^{\beta [u]}-1),
\end{equation}
for some constants $\sigma_{m0}, K,$ and $\beta$, and used the model \eqref{eq:Divergence}-\eqref{eq:BoundaryVal} and \eqref{eq:IvorraMir} as a boundary-value problem for an elliptic equation with nonlinear transmission conditions at the membrane.

	\bigskip
The classical and more involved model for $\sigma_m$ due to Neu and Krassowska \cite{NeuKra99} is explained in the following. It assumes that $\sigma_m$ is the sum of $\sigma_{m0}$ and an electropermeabilization current. The latter is proportional to the
 pore density $N$, which in turn is governed by an ordinary differential equation:
\begin{align}
	\label{eq:electropermeabilization}
	\sigma_m([u],t) &= \sigma_{m0} + \beta N([u],t)  &\text{on }\Gamma \times ]0,T[,\\
	\label{eq:InitialN}
	N([u],0) &= N_0(x) &\text{on }\Gamma, \\
\label{eq:PoreDensity}
	\partial_t N([u],t) &=  \alpha\; e^{\left(\frac{[u](x,t)}{V_{ep}}\right)^2}\left(1-\frac{N([u],t)}{N_0}e^{-q\left(\frac{[u](x,t)}{V_{\mathrm{ep}}}\right)^2}\right) &\text{on }\Gamma  \times ]0,T[,
\end{align}
where $ \alpha,\beta, q,$ and $N_0$ are constants, $V_{\mathrm{ep}}$ is the minimum transmembrane voltage for electropermeabilization, and $T$ is the final time. 

Given the condition
\begin{align}
\label{eq:Initial}
u(x,t)&= u_{\mathrm{ref}}&\text{on } Y \text{ for } t<0,
\end{align} the initial value problem \eqref{eq:Divergence}-\eqref{eq:BoundaryVal} and \eqref{eq:electropermeabilization}-\eqref{eq:Initial} is then solved on $Y\times ]0,T[$.

\medskip

Another model for $\sigma_m$ has been developed in \cite{KavLegPoiWey14}. Together with \eqref{eq:electropermeabilization}, \eqref{eq:InitialN}, one uses the dynamics
\ben
\partial_t N([u],t) &= \max \left( \frac{\beta([u])-N([u],t)}{\tau_{\mathrm{ep}}} , \frac{\beta([u])-N([u],t)}{\tau_{\mathrm{res}}} \right)
\een
with  
\[
	\beta(\lambda)=(1+\tanh(k_{\mathrm{ep}}(|\lambda|-V_{\mathrm{ep}})))/2,
\]
and given constants $\tau_{\mathrm{ep}}$, $\tau_{\mathrm{res}}$, and $k_{\mathrm{ep}}$.

\section{Wellposedness of the electropermeabilization model}
\label{sec:Wellposed}
In this section, we treat the classical electropermeabilization model model \eqref{eq:Divergence}-\eqref{eq:BoundaryVal} and \eqref{eq:electropermeabilization}-\eqref{eq:Initial} and study it in the form of an ODE on the membrane $\Gamma$.

As a preliminary step, let us prove the following representation of the pore density $N$.

\begin{lem}
\label{prop:PoreDensity}
%\corr{Spaces are a topic to be adapted in the whole of the draft. Here we might either be more general or more specific.}
\begin{enumerate}
\item[(i)] For $[u]=v$, the solution of the initial value problem in \eqref{eq:InitialN}, \eqref{eq:PoreDensity} is
\begin{equation}
\label{eq:PoreDensityRep}
N(x,t) =e^{-\int_0^t \frac{\alpha}{N_0}\; e^{(1-q)\left(\frac{v(x,\tau)}{V_{\mathrm{ep}}}\right)^2 } d\tau}\; N_0 + \left( \int_0^t \alpha\; e^{\left(\frac{v(x,s)}{u_{0}}\right)^2} \, e^{-\int_s^t \frac{\alpha}{N_0}\; e^{(1-q)\left(\frac{v(x,\tau)}{V_{\mathrm{ep}}}\right)^2 } d\tau}\,ds  \right).
\end{equation}

\item[(ii)] The pore density $N$, considered as a mapping $v(x,t)\mapsto N(v(x,t),t)$
\beql{eq:Leguebe}
	C([0,T],C(\Gamma))\times[0,T]\to C([0,T],C(\Gamma)),
\eeq
maps bounded sets to bounded sets.
\end{enumerate}
\end{lem}

\begin{proof}
Note that the solution to a linear inhomogeneous ordinary differential equation
\begin{equation}
\label{eq:AllgForm}
	\frac{\partial}{\partial t}N(t) = A(t)\, N(t) + b(t)
\end{equation}
is given by \cite[Thm. 5.14]{Ama90}
\begin{equation}
\label{eq:AllLoes}
	N(t) = U(t,0)\; N_0 + \int_s^t U(t,s) \;b(s) \; ds,
\end{equation}
where 
\[
	U(t,s) = \int_s^t A(\tau)\,d\tau.	
\]
Equation \eqref{eq:PoreDensity} is a special form of \eqref{eq:AllgForm}, and the coefficients $A$ and $b$ are
\[
	A(t) = -\frac{\alpha}{N_0}\; e^{(1-q)\left(\frac{[u](t)}{V_{\mathrm{ep}}}\right)^2 },
\]
and
\[
	b(t)= \alpha\; e^{\left(\frac{[u](t)}{V_{\mathrm{ep}}}\right)^2}.
\]
Inserting $A$ and $b$ into the general solution \eqref{eq:AllLoes}, we directly obtain the representation \eqref{eq:PoreDensityRep} 
in (i).

Using the norm $\| v\|_{C(\Gamma)}=\sup_{x\in \Gamma}|v(x)|$, the boundedness property in (ii) is then immediate.
\end{proof}

\begin{remark}
\label{rem:N}
In practice, it is clear that the potential $v$ stays finite. One may therefore choose a  real number $M>0$ and work instead of $N(v,t)$ with the function
\beql{eq:Nmod}
N_M(v,t) := N(v_M,t)\qquad\text{with } v_M:=\left\{\begin{array}{ll}
|v|&|v|\leq M\\
M&|v|>M\\
-M&|v|<M
\end{array}\right..
\eeq
For $\|v\|_{L^\infty(\Gamma)}<M$, this cutoff preserves the pore density: $N_M(v,t)=N(v,t)$. In Lemma \ref{lem:Lipschitz}, it is shown that the function $v\mapsto N_M(v,t)v_M$, considered in $C((0,T);L^2(\Gamma))$, has a global Lipschitz property.
\end{remark}

\subsection{Reduction to an ordinary differential equation}
\begin{defin}[Stekhlov-Poincar\'e operators]
\label{def:StekhlovP}
Let $H^{s}(\Gamma)$ be the standard Sobolev space on $\Gamma$ of order $s$. 
Let $f\in H^{\frac{1}{2}}(\Gamma)$ be given. Define solutions of Dirichlet boundary value problems and assign the Neumann data via the \emph{Stekhlov-Poincar\'e operators} $\Lambda_c$, $\Lambda_e$:
$H^{1/2}(\Gamma) \rightarrow H^{-1/2}(\Gamma)$  and $ \Lambda_0$: $H^{1/2}(\partial \Omega) \rightarrow H^{-1/2}(\Gamma)$, 

$$
\Lambda_c f := \partial_n P_1, \quad \Lambda_e f := \partial_n P_2, \quad \Lambda_0 f := \partial_n P_3,$$
where $P_i, i=1,2,3$ are solutions to
\[
\left\{
\begin{array}{lcll}
\Delta P_1&=&0 \quad&\mbox{in } Y_i, \\
 P_1  &= &f \quad &\mbox{on } \Gamma, \\
\end{array} \right. 
\;\; 
\]
and 
\[
\left\{
\begin{array}{lcll}
\Delta P_2 &=&0 \quad&\mbox{in } Y_e, \\
P_2 &=&0 \quad &\mbox{on } \partial Y, \\
P_2 &= &f  \quad &\mbox{on } \Gamma,
\end{array} \right.
\qquad \qquad
\left\{
\begin{array}{lcll}
	\Delta P_3 &=&0 \quad&\mbox{in } Y_e, \\
	P_3 &=&f  \quad &\mbox{on } \partial Y, \\
	P_3 &=& 0 \quad &\mbox{on } \Gamma.\\
\end{array} \right. 
\]
\end{defin}
The following results hold.
\begin{lem}
\label{prop:reduction}
\begin{enumerate}
\item[(i)]
Solving the problem
\eqref{eq:Divergence}-\eqref{eq:BoundaryVal} and \eqref{eq:electropermeabilization}-\eqref{eq:Initial} 
for $u=(u_i,v,u_e)$ on $Y_i\cup\Gamma\cup Y_e$ is equivalent to solving
the initial value problem
\begin{equation}
\label{eq:redODE}
\begin{aligned}
	\frac{c_m}{\delta}\partial_t v + \frac{\sigma_m}{\delta}(v,t)v&+\Lambda_c B^{-1}v = G,\\
	v(0)&= \varphi,
\end{aligned}
\end{equation}
for $v$ on $\Gamma$, with the correspondence
\[
\begin{aligned}
	u_i&= - B^{-1}(v + \Lambda_e^{-1}\Lambda_0 g),\\
	u_e &= u_i + v.
\end{aligned}
\]
Here, $B=Id + \Lambda_e^{-1}\Lambda_0 , G= - \Lambda_c B^{-1} \Lambda_e^{-1} \Lambda_0 g$, and
\begin{equation}
\label{eq:sigExpl}
\sigma_m(v,t) = \sigma_{m0}(x) + \beta N(v,t).
\end{equation}

\item[(ii)] The linear operator $\Lambda_c B^{-1}: H^1(\Gamma)\to L^2(\Gamma)
$ is m-accretive. In particular, one has
\beql{eq:accret}
	\forall v:\qquad\langle \Lambda_c B^{-1}v,v\rangle_{L^2}\geq 0,
\eeq
where $\langle \; , \; \rangle_{L^2}$ is the  scalar product on $L^2(\Gamma)$.
\end{enumerate}
\end{lem}
\begin{proof}
The reduction of the time-dependent model on $\Omega$ to the initial value problem on $\Gamma$ in \eqref{eq:redODE}, using the Steklov-Poincar\'e operators, is the same as in \cite[Lemma~9]{KavLegPoiWey14}. The property in (ii) is shown in \cite[Lemma~8]{KavLegPoiWey14}.
\end{proof}

For establishing existence and uniqueness results (in Theorem~\ref{thm:WellPosed}), we use the following lemma on the Lipschitz property of the function $N_M$ introduced in Remark~\ref{rem:N}.

\begin{lem}
\label{lem:Lipschitz}
Let $M>0$, and let $N_M(v,t) = N(v_M,t)$ with $v_M=\sgn(v)\min(|v|,M)$ be the modified pore density defined by \eqref{eq:Nmod}. Then 
\[
	v\longmapsto N_M(v,t)v_M
\]
is  global Lipschitz in $C((0,T);L^2(\Gamma))$.
\end{lem}
\begin{proof}
Let $v_1$, $v_2\in C((0,T);L^2(\Gamma))$. One has the algebraic identity
\begin{multline}
N(v_{1M},t)v_{1M}-N(v_{2M},t)v_{2M}= \\
\bigl(N(v_{1M},t)v_{1M}-N(v_{1M},t)v_{2M}\bigr) + \bigl(N(v_{1M},t)v_{2M}-N(v_{2M},t)v_{2M}\bigr). \label{thatident}
\end{multline}
Using the boundedness of $v_M$, (\ref{thatident}) shows that it suffices to prove that $N(v_M,t)$ is global Lipschitz in  $C((0,T);L^2(\Omega))$.

Consider the explicit form of $N(v,t)$ in \eqref{eq:PoreDensityRep}. As $\|v_M\|_{L^{\infty}}\leq M$, there exists a constant $L(M)$ such that
\beq
|N(v_{1M},t)-N(v_{2M},t)|^2\leq L(M) \int_0^t |v_1(x,s)-v_2(x,s)|^2ds.
\eeq
Therefore, we have
\beq
\|N(v_{1M},t)v_{1M}-N(v_{2M},t)v_{2M}\|_{C((0,T);L^2(\Gamma))}\leq C(M) \|v_{1}-v_{2}\|_{C((0,T);L^2(\Gamma))},
\eeq
and the global Lipschitz property of $N_M$ in $C((0,T);L^2(\Omega))$ holds.
\end{proof}

Using Lemma \ref{lem:Lipschitz}, we now come to the well-posedness results. For this end, we introduce the following auxiliary problem. As a variant to \eqref{eq:BoundaryVal}, we consider
\begin{align}
\label{eq:BoundaryValNeu}
	\sigma(x)\b n\cdot\nabla u(x,t) &= \frac{c_m}{\delta}\partial_t [u](x,t) + \frac{\sigma_m\left( [u]_M(x,t),t\right) }{\delta}[u]_M(x,t) &\text{ on }\Gamma.
	\tag{4'}
\end{align}

Using the same procedure as in Lemma \ref{prop:reduction}, we find that the model \eqref{eq:Divergence}-\eqref{eq:BoundaryEq},\eqref{eq:BoundaryValNeu} and \eqref{eq:electropermeabilization}-\eqref{eq:Initial} 
is equivalent to solving
\begin{equation}
\label{boundedv}
\begin{aligned}
 \frac{c_m}{\delta}\partial_t \tilde v + \frac{\sigma_m(\tilde   v_M,t)}{\delta} \tilde v_M +\Lambda_c B^{-1}\tilde v = G , \\
\tilde v(0)=\varphi . 
\end{aligned}
\end{equation}

Let us now state the well-posedness properties of our initial value problems on $\Gamma$.
\begin{theorem}
\label{thm:WellPosed} Let $G \in C^1((0,T);H^1(\Gamma))$ and $\varphi\in H^2(\Gamma)$. 
\begin{itemize}
\item[(i)]
The initial value problem in \eqref{boundedv} has a unique global solution $\tilde v\in C([0,T];H^2(\Gamma))$.
\item[(ii)] For the initial value  problem \eqref{eq:redODE}, there is a $t_0>0$ such that there exists a solution $v\in C([0,t_0[;H^2(\Gamma))$.
\item[(iii)] The solution in (ii) is unique on $C([0,t_1],H^2(\Gamma))$ for any  closed interval $[0,t_1]\subset[0,t_0[$.
\end{itemize}
\end{theorem}
\begin{proof}
(i): Let $M>\|\varphi\|_{L^{\infty}}$ be a constant and consider  the initial value problem \eqref{boundedv}. Fix a number $T>0$.

Due to the global Lipschitz property of $N_Mv_M$ shown in Lemma \ref{lem:Lipschitz}, one can apply the fixed point argument in \cite[Thm.10]{KavLegPoiWey14}) to conclude that there exists a unique solution $\tilde v \in C([0,T];L^2(\Gamma))$ solving (\ref{boundedv}).

If one additionally assumes that $G \in C^1([0,T];H^1(\Gamma))$ and $\varphi\in H^2(\Gamma)$, then one can likewise conclude $\tilde v\in C^1([0,T];H^2(\Gamma))$.
Then we have that $\partial_n u_i  \in L^2(\Gamma)$. With such boundary regularity, we infer $\tilde u_i \in H^{3/2}(Y_i)$, similarly $\tilde u_e\in H^{3/2}(Y_e)$. Then $\tilde v=\tilde u_e-\tilde u_i \in C([0,T];H^1(\Gamma))$. 
Using this argument once again, we have that $\tilde v=\tilde u_e-\tilde u_i \in C([0,T];H^2(\Gamma))$.
\medskip

(ii): We will now show that the solution $\tilde v$ to \eqref{boundedv} found in point (i) solves locally the original problem \eqref{eq:redODE}. -- Using the Sobolev embedding theorem one has  that
\[
	\Lambda_c B^{-1}\tilde v \in C([0,T];H^1(\Gamma)) \hookrightarrow C([0,T];C(\Gamma)).
\]
Take a constant $C_M$ such that, for any $t\leq T$, one has
\[
	\left\| \frac{\sigma_m(\tilde v_M,t)}{\delta} \tilde v_M +\Lambda_c B^{-1}\tilde v + G\right\|_{C(\Gamma)} \leq C_M.
\]
Define $$t_0:=\frac{c_m}{\delta}\frac{M-\|\varphi\|_{L^{\infty}}}{C_M}.$$ Then, for $t\leq t_0$, one gets
\[
\begin{aligned}
	\|\tilde v(x,t)\|_{L^{\infty}(\Gamma)} & \leq \|\varphi\|_{L^{\infty}} + t C_M , \\
	&\leq M.
	\end{aligned}
\]
But for $\|\tilde v\|_\infty<M$, one has that $\tilde v_M=\tilde v$ and $N_M(\tilde v,t)=N(\tilde v,t)$. Therefore, the expressions in \eqref{eq:redODE} and \eqref{boundedv} are the same, which implies that, locally, $\tilde v$ solves as well the original initial value problem  \eqref{eq:redODE}.

\medskip

(iii): Take two solutions $v,w$ to \eqref{eq:redODE} in $C^1([0,t_1],H^2(\Gamma))$. Due to closedness of $[0,t_1]$ and continuity of the norm $\| .\|_{H^2}\to \mathbb{R}$, there exists a $M>0$ such that for every $t\in[0,t_1]$, one has
\[
	\|v(t)\|_{H^2}<M\qquad\text{and}\qquad \|w(t)\|_{H^2} < M.
\]
But then the cutoff with respect to $M$ does not change the functions: $v_M=v$ and $w_M=w$. Therefore, $v$ and $w$ also solve \eqref{boundedv}. But for that ODE, one has a global uniqueness property. Therefore $v=w$ on $[0,t_1]$.
\end{proof}

We now give a more detailed analysis of the terms in equation \eqref{eq:redODE} to show that a solution cannot blow up in finite time (see Theorem \ref{thm:noBlowup}).

Note that for $\sigma_m$ given by \eqref{eq:electropermeabilization}, there exists a $C\in \mathbb{R}$ such that one has for all $v$ that
\begin{equation} \label{lem:Coercivity}
	\langle\sigma_m(v,t)v,v\rangle_{L^2}\geq C \| v \|_{L^2}^2.
\end{equation}
This immediately follows from the expression of the membrane conductivity in \eqref{eq:electropermeabilization} and the fact that both the pore density $N$ as well as $N_M$ in \eqref{eq:Nmod} are positive.

\begin{theorem}
\label{thm:noBlowup}
For a function $v\in C^1([0,t_0[,L^2(\Gamma))$ which solves \eqref{eq:redODE}, it is impossible that
\[
	\| v(t_k) \|_{L^2(\Gamma)}\longrightarrow_{t\to b}\infty\qquad \text{for } b\in[0,t_0[.
\]
\end{theorem}

\begin{proof}
Take as an indirect assumption a blow up-sequence $\| v(t_k) \|_X\to\infty$ with $t_k\to b$. Without loss of generalization, we may choose $t_k\in [0,t_0[\;\cap \;W$, where $W$ is a neighborhood of $b$ such that $v$ is nonzero on $[0,t_0[\;\cap \;W$. Due to the $C^1$-regularity property of $v(t)$ and $v\neq 0$, the function
\[
	[0,t_0[\;\cap\; W\to\mathbb{R}:\qquad t\longmapsto\| v(t)\|_{L^2}
\]
is then continuously differentiable.

The sequence $t_k\to b(x)$ having the Cauchy property, the slope of the secants satisfies 
\[
	\frac{\left| \|v(t_{k+1})\|_-\|v(t_k)\|\right|}{t_{k+1}-t_k}\longrightarrow\infty,
	\]
as well. We then will work with a sequence $\tau_k$ such that
\beql{eq:DerUp}
	\partial_t \| v(\tau_k)\|_{L^2}\longrightarrow\infty,
\eeq
chosen by the mean-value theorem.

\bigskip

Consider equivalently to \eqref{eq:redODE} the equation
\[
	\sigma(v)v=G-C_m\partial_t v - \Lambda_c B^{-1}v.
\]
Take the $L^2$-scalar product with $v$ and take account of $\langle \partial_t v,v\rangle=\|v\|\;\partial_t\| v\|$. Then estimate the right-hand side with the Cauchy-Schwarz inequality and the accretivity property \eqref{eq:accret}:
\begin{align*}
	\langle\sigma(v)v,v\rangle_{L^2}
	&=
	\langle G,v\rangle_{L^2} -C_m\langle\partial_t v,v\rangle_{L^2} - \langle\Lambda_c B^{-1}v,v\rangle_{L^2} , \\
	&\leq \| G\|_{L^2} \| v \|_{L^2} - C_m\|v \|_{L^2} \;\partial_t\| v\|_{L^2}.
\end{align*}
Divide by $\| v \|_{L^2}$ to find
\beql{eq:CrucialIneq}
	\frac{\langle\sigma(v)v,v\rangle_{2}}{\| v\|_{L^2}}\leq \| G \|_{L^2} - C_m\partial_t \| v\|_{L^2}.
\eeq
From (\ref{lem:Coercivity}), we already know that the left-hand side stays positive.

Evaluate then expressions in inequality \eqref{eq:CrucialIneq} for the sequence $\tau_k$ in \eqref{eq:DerUp}. The result is that the right-hand side would tend to $-\infty$, which is impossible. This shows that no blow up of $v$ in $L^2$ can occur.
\end{proof}

\section{Homogenization}
\label{sec:Homogenization}

Let $\Omega$ be a bounded domain in $\mathbb{R}^2$, which carries a periodic structure made up by periodic open sets $\varepsilon Y$. The reference domain $Y=Y_i\cup Y_e\cup\Gamma$ contains a cell inside with membrane $\Gamma$, where $Y_i$ is the intracellular domain and $Y_e$ is the extracellular domain. The whole domain $\Omega$ is thus composed of
\[
	\Omega = \Omega^+\cup\Omega^-\cup\Gamma_{\varepsilon},
\]
where $\Omega^+$ is the collection of extracellular domains, $\Omega^-$ is the collection of intracellular domains and $\Gamma_{\varepsilon}$ is the collection of membranes.

We write the thickness of the membrane of the cells $\varepsilon Y$ in the form
\[
	\delta=\varepsilon\; \delta_0,
\]
where $\epsilon$ is the scale of the cell and $\delta_0$ is the reference cell membrane thickness for $Y$. 

As in \cite{laure}, we want to study behavior of the electrical field on this cell cluster and recover  features of the microscopic cell model from tissue measurements. Considering the cell model in \eqref{eq:Divergence}-\eqref{eq:BoundaryVal} and \eqref{eq:electropermeabilization}-\eqref{eq:Initial} for a domain $Y$, we first give the  model equation for $u_\varepsilon$ in $\Omega$:

\beql{eq:PeriodicSystem}
\begin{aligned}
\nabla\cdot(\sigma(x)\nabla u_{\varepsilon}(x,t)) &= 0&  \text{ }\text{in } \Omega^+ ,\\
\nabla\cdot(\sigma(x)\nabla u_{\varepsilon}(x,t)) &= 0&  \text{ }\text{in } \Omega^- ,\\
[\sigma \nabla u_{\varepsilon} \cdot\textbf{n} ] &= 0& \text{ } \text{on } \Gamma_{\varepsilon}, \\
\frac{c_m}{\delta}\frac{\partial}{\partial t} [u_{\varepsilon}] + \frac{1}{\delta}\sigma_m([u_{\varepsilon}]_M,t) [u_{\varepsilon}]_M &= \sigma {\partial_n u^-_{\varepsilon}}&  \text{ }\text{on } \Gamma_{\varepsilon}, \\
\ [u_{\varepsilon}](x,0) &= S_{\varepsilon}& \text{ } \text{on } \Gamma_{\varepsilon}, \\
\ u_{\varepsilon}(x,t) &= 0& \text{ } \text{on } \partial\Omega, \\
\end{aligned}
\eeq
where $S_{\varepsilon}(x)=\varepsilon S_1(x,\frac{x}{\varepsilon})+R(\varepsilon)$ and $\sigma_m= \sigma_{m0} + \beta N([u_{\varepsilon}],t)$. The pore density
$N([u_{\varepsilon}],t)$ is governed by \eqref{eq:PoreDensity}.

Here, in the second equation on $\Gamma_\varepsilon$, the quantity $[u_\varepsilon]_M$ is understood in the sense of the definition in  \eqref{eq:Nmod}, i.e., $[u_\varepsilon]_M=\sgn([u_\varepsilon])\min(|[u_\varepsilon]|,M)$ for a constant $M>0$.

\medskip

Given the physical observation that the voltage $v$ stays bounded, it is reasonable that for proper $M>0$, the system \eqref{eq:PeriodicSystem} is an accurate model for the real potential. Given Lemma~\ref{prop:reduction} and Theorem \ref{thm:WellPosed}, it is also well-posed.

\medskip

We want to explore the limit of the solution $u_\varepsilon$ as $\varepsilon\to 0$. For this end, we start with an energy estimate on the solution $u_\varepsilon$ which will be needed later when investigating the limit.

\begin{prop}
\begin{itemize}
\item[(i)] We have for $u_\varepsilon$ in \eqref{eq:PeriodicSystem}  the energy estimate
\beql{prop:Energy}
	\int_0^t\int_\Omega\sigma|\nabla u_\varepsilon|^2dx\;dt + \frac{1}{\varepsilon}\int_{\Gamma_\varepsilon}[u_\varepsilon]^2(x,t)dS\leq C.
\eeq
\item[(ii)] In particular, the estimate
\beql{eq:EstimateTMV}
	\int_{\Gamma_\varepsilon}[u_\varepsilon]^2 dS\leq C\varepsilon
\eeq
holds.
\end{itemize}
\end{prop}

\begin{proof}
Multiply (\ref{eq:PeriodicSystem}) by $u_{\varepsilon}$, then integrate by parts to find
\begin{multline}
\int_0^t\int_\Omega\sigma|\nabla u_\varepsilon|^2dx\;dt + \frac{\alpha}{2\varepsilon}\int_{\Gamma_\varepsilon}[u_\varepsilon]^2(x,t)dS \\
+
\frac{1}{\varepsilon}\int_0^t\int_{\Gamma_\varepsilon}\sigma_m([u_{\varepsilon}]_M,\tau)[u_\varepsilon][u_\varepsilon]_M(x,t)dS\;dt =
\frac{\alpha}{2\varepsilon}\int_{\Gamma_\varepsilon}[S_\varepsilon]^2(x)dS.
\end{multline}
The statement is then derived from the fact that 
\[
	\sigma_m [u_\varepsilon][u_\varepsilon]_M \geq 0
\]
and $S_{\varepsilon}(x)=\varepsilon S_1(x,\frac{x}{\varepsilon})+o(\varepsilon)$.
\end{proof}

\bigskip

For now, let us formally assume that the solution $u_{\varepsilon}$ of \eqref{eq:PeriodicSystem} has the form
\beql{eq:AnsatzSpace}
u_{\varepsilon} ( x,t) = u_0( x,t) + \varepsilon u_1(x,\frac{ x}{\varepsilon},t)+ o(\varepsilon).
\eeq
We will calculate the equation for $u_0$ in Subsection~\ref{subsec:HomFormal} and then prove rigorously that $u_\varepsilon$ converges in an appropriate sense to $u_0$ in Subsection~\ref{subsec:HomConverg}.

\subsection{Formal calculation of the homogenization limit}
\label{subsec:HomFormal}

To find the precise form of the terms in the ansatz \eqref{eq:AnsatzSpace}, we can apply the arguments developed in \cite{AmaAndBisGia04}. For this end, it is required that for the membrane conductivity one has that
\[
	\sigma_m(0,t)=\mathrm{ constant}.
\]
(see \cite[Secs. 3.2 and 3.3]{AmaAndBisGia04}). This condition can be ensured for the model \eqref{eq:electropermeabilization}, together with \eqref{eq:PoreDensity}: From (\ref{eq:PoreDensityRep}), one can prove that $N(0,t) = N_0$, and therefore $\sigma_m(0,t)= \mathrm{ constant}$.
% For the model in \eqref{eq:Leguebe}, one could as well ensure this property for $\beta(0)=N_0$, $N(0,t) = N_0$.)

\bigskip

Before calculating the limit, we first give some definitions. 
Introduce the transform
\[
	T: H^{1/2}(\Gamma)\to C([0,T],H_p^1(Y)),
\]
where
\[
	H_p^1(Y)= \bigg\{ u \text{ } is \text{ } periodic \text{ } in \text{ } Y: u|_{Y_i}\in H^1(Y_i) \text{ and  } u|_{Y_e}\in H^1(Y_e), \text{ } \int_{Y}u=0 \bigg\},
\]
by
\[
	T(s)(y,t):= v(y,t)
\]
with $v$ being the solution to the following system with boundary data $s$:
\begin{align*}
\nabla\cdot(\sigma(x)\nabla v) &= 0& \text{ } \text{in } Y_i,\\
\nabla\cdot(\sigma(x)\nabla v) &= 0& \text{ } \text{in } Y_e,\\
[\sigma\nabla v \cdot\textbf{n} ] &= 0& \text{ } \text{on } \Gamma, \\
\frac{c_m}{\delta_0}\frac{\partial}{\partial t} [v] + \frac{1}{\delta_0}\sigma_m( 0,t) [v] &= \sigma {\partial_n v^-} & \text{ } \text{on } \Gamma, \\
\ [v](x,0) &= s & \text{ } \text{on } \Gamma .
\end{align*}
We define next the cell problems $\chi^0 : \Omega \rightarrow \mathbb{R}^d$ and $\chi^1 : \Omega \times (0,T) \rightarrow \mathbb{R}^d$. For this, let $\b e_h$ be the $h$-{th} unit vector in $\mathbb{R}^d$. Then the component $\chi^0_h \in H^1_p(Y)$ satisfies
\begin{align*}
\nabla\cdot(\sigma(x)\nabla \chi^0_h) &= 0 & \text{in } Y_i,\\
\nabla\cdot(\sigma(x)\nabla \chi^0_h) &= 0 & \text{in } Y_e,\\
~[\sigma (\nabla_y\chi^0_h-\b e_h)\cdot \b n]  &= 0 & \text{on } \Gamma, \\
\ [\chi^0_h](x,0) &= 0 & \text{on } \Gamma.
\end{align*}
The component
$\chi^1_h$ is defined by
\beq
\chi^1_h = T(\sigma (\nabla_y\chi^0_h-\b e_h)\cdot \b n).
\eeq
By a calculation analogous to \cite[Sec.3]{AmaAndBisGia04}, one finds that the candidate $u_0$ in equation~\eqref{eq:AnsatzSpace} satisfies
%\corr{The convention is to write everything which is a vector or a point in $\mathbb{R}^d$ in boldface, everything which is a scalar in normal letters.}
\beql{eq:FormLimit}
\dive \left[ -\sigma_0\nabla_{x} u_0 - A^0 \nabla_{x} u_0 - \int_0^t A^1(t-\tau)\nabla_{x} u_0(x,\tau)d\tau + \b F(x,t)\right] =0 .
\eeq
Here, the matrices $A^0$, $A^1,$ and $\b F(x,t)$ are defined by
\beql{eq:DefA_F}
\left\{
\begin{aligned}
\sigma_0 &= \sigma_1|Y_i| + \sigma_2|Y_e|,& \\
(A^0)_{jh} &= \int_{\Gamma} [\sigma] \chi^0_h n_j dS,&\\
(A^1)_{jh} &= \int_{\Gamma} [\sigma \chi^1_h] n_j dS,&\\
\b F &= \int_{\Gamma} [\sigma T(S_1(x,\cdot))](y,t)\;\b n\;dS,&
\end{aligned}
\right.
\eeq
where $\sigma_i=\sigma|_{Y_i}$ and $\sigma_e=\sigma|_{Y_e}$,
with $\chi_h^0, \chi_h^1$ and $T$ given above.

\subsection{Convergence}
\label{subsec:HomConverg}
While in Subsection \ref{subsec:HomFormal}, we derived the formal limit \eqref{eq:FormLimit} for the ansatz of the asymptotic expansion \eqref{eq:AnsatzSpace}, we now state its convergence properties.

\begin{theorem}
\label{thm:Convergence}
For the periodic solution $u_\varepsilon$ in~\eqref{eq:PeriodicSystem} and the homogenized solution $u_0$ in~\eqref{eq:FormLimit}, we have the convergence \[
	u_\varepsilon\to u_0
\]
weakly in $L^2([0,T]\times\Omega)$ and strongly in $L^1_{\mathrm{loc}}([0,T],\Omega)$.
\end{theorem}

The proof relies on arguments developed in \cite{AmaAndBisGia04}. For the sake of a  readability, we outline them in the appendix, and only prove here the crucial lemma needed for their adaption to our case.

\begin{lem}
\label{lem:K3}
%For the term $K_{3\varepsilon}$ in \eqref{eq:DefinitionK}, one has that
%\[
%	\lim_{\varepsilon\to 0}K_{3\varepsilon}=0.
%\]
For $M>1$, there exists a constant $C(M)$ such that
\beql{eq:AuxEstimateK3}
	\int_0^T\!\int_{\Gamma_\varepsilon} |\sigma_m( 0,t)[u_\varepsilon]-\sigma_m([u_\varepsilon]_M,t)\;[u_\varepsilon]_M\;dS\;| dt\leq C\varepsilon.
\eeq
\end{lem}
\begin{proof}
We have 
\ben
& &\sigma_m( 0,t)[u_\varepsilon]-\sigma_m([u_\varepsilon]_M,t)\;[u_\varepsilon]_M \\
&= &\sigma_m( 0,t)[u_\varepsilon]-\sigma_m([u_\varepsilon]_M,t)\;[u_\varepsilon] + \sigma_m([u_\varepsilon]_M,t)\;([u_\varepsilon]-[u_\varepsilon]_M).
\een
By the explicit form of $N(v,t)$ in \eqref{eq:PoreDensityRep} and $|v_M|_{L^\infty}\leq M$, there exists a constant $L(M)$ such that
\beq
|N([u_\varepsilon]_M,t)-N(0,t)|^2\leq L(M) \int_0^t [u_\varepsilon]^2_Mds,
\eeq
and $\sigma_m([u_\varepsilon]_M,t)\leq C(M)$.

Together with the fact that $\left|\int_0^T [u_\varepsilon]-[u_\varepsilon]_M \;ds\right| \leq \int_0^T [u_\varepsilon]^2 ds$, we can thus 
conclude that
\[
	\int_0^T\!\int_{\Gamma_\varepsilon} \bigl|\sigma_m( 0,t)[u_\varepsilon]-\sigma_m([u_\varepsilon]_M,t)\;[u_\varepsilon]_M\bigr|\;dS\; dt\leq C(M) \varepsilon.
\]
The lemma then follows by the energy estimate \eqref{eq:EstimateTMV}.
\end{proof}

\begin{table}
\centering
\begin{tabular}{ l c r }
\hline
  Symbol & Value & Definition \\
  \hline
  $\sigma_i$  &  0.455  & intracellular conductivity \\
  $\sigma_e$  &  5   &  extracellular conductivity  \\
  $L$ & $2\times 10^{-4}$   &  computation domain size \\
  $r$ & $0.5\times 10^{-4}$  &  cell radius \\
  $\delta$ & $5\times 10^{-9}$  & membrane thickness \\
  $r_p$ & 0.76 & pore radius \\
  $\sigma_p$ &  0.0746 &  pore conductivity \\
  $V_{\mathrm{ep}}$     &   0.258  &  characteristic voltage of electropermeabilization \\
  $\alpha$  &  $10^9 $   &   electropermeabilization parameter   \\
  $N_0$    &  $1.5\times 10^{9}$   &   equilibrium pore density  \\
  $c_m$   &   $9.5\times10^{-12}$ &   membrane capacitance\\
  \hline
\end{tabular}
\vspace{0.5cm}
\caption{Model parameters used for the numerical computations.}
\label{tab:Parameters}
\end{table}

\section{Numerical experiments}
\label{sec:Illustration}
In the preceding section, we have modeled macroscopic processes as homogenized quantities with specific effective material parameters. In this section we show the sensitivity of the effective parameters to microscopic properties relevant in electropermeabilization.

We use FEM with mesh generator \cite{PerStr04} to implement all the numerical simulations. We present the numerical experiments from two aspects: First we will simulate the single cell model \eqref{eq:redODE} and show the electropermeabilization at cell level. Next we show how the microscopic parameters affect effective parameters and anisotropy properties in the homogenized model \eqref{eq:FormLimit}.

\subsection{Electropermeabilization simulation for a single cell}
We simulate the single cell model \eqref{eq:redODE} in a square domain $[0,L]\times[0,L]$, the cell is a circular in the center of the square with cell radius $r$. The parameter $\beta$ in \eqref{eq:PoreDensity} is given by 
\beq
\beta=\frac{2\pi r^2_p \sigma_p \delta}{\pi r_p+2\delta}.
\eeq
All the parameters are given in Table \ref{tab:Parameters}. Figure \ref{fig:CellModel} shows the results for the time evolution and the voltage after $2~\mu s'$.

\begin{figure}
\centering
\subfigure{\includegraphics[width=7.5cm]{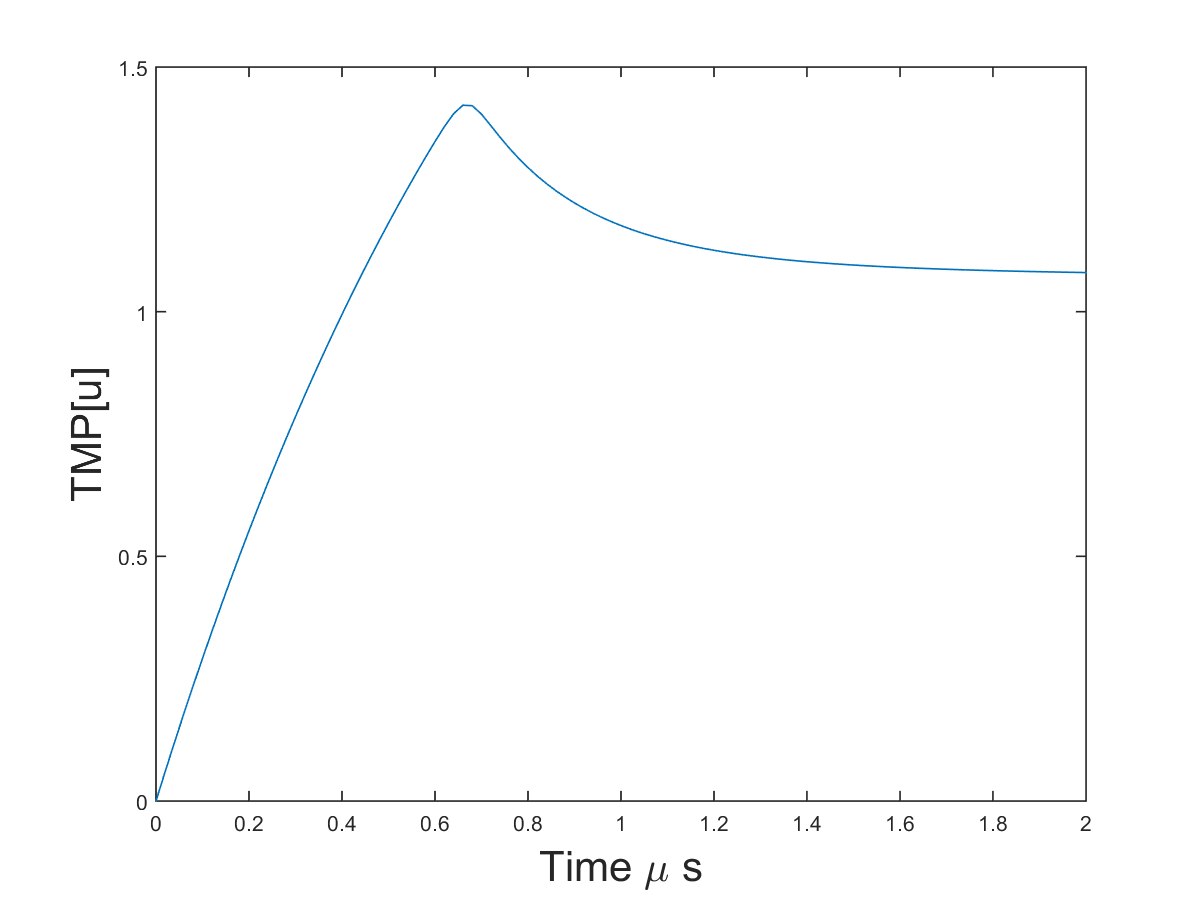}}
\subfigure{\includegraphics[width=7.5cm]{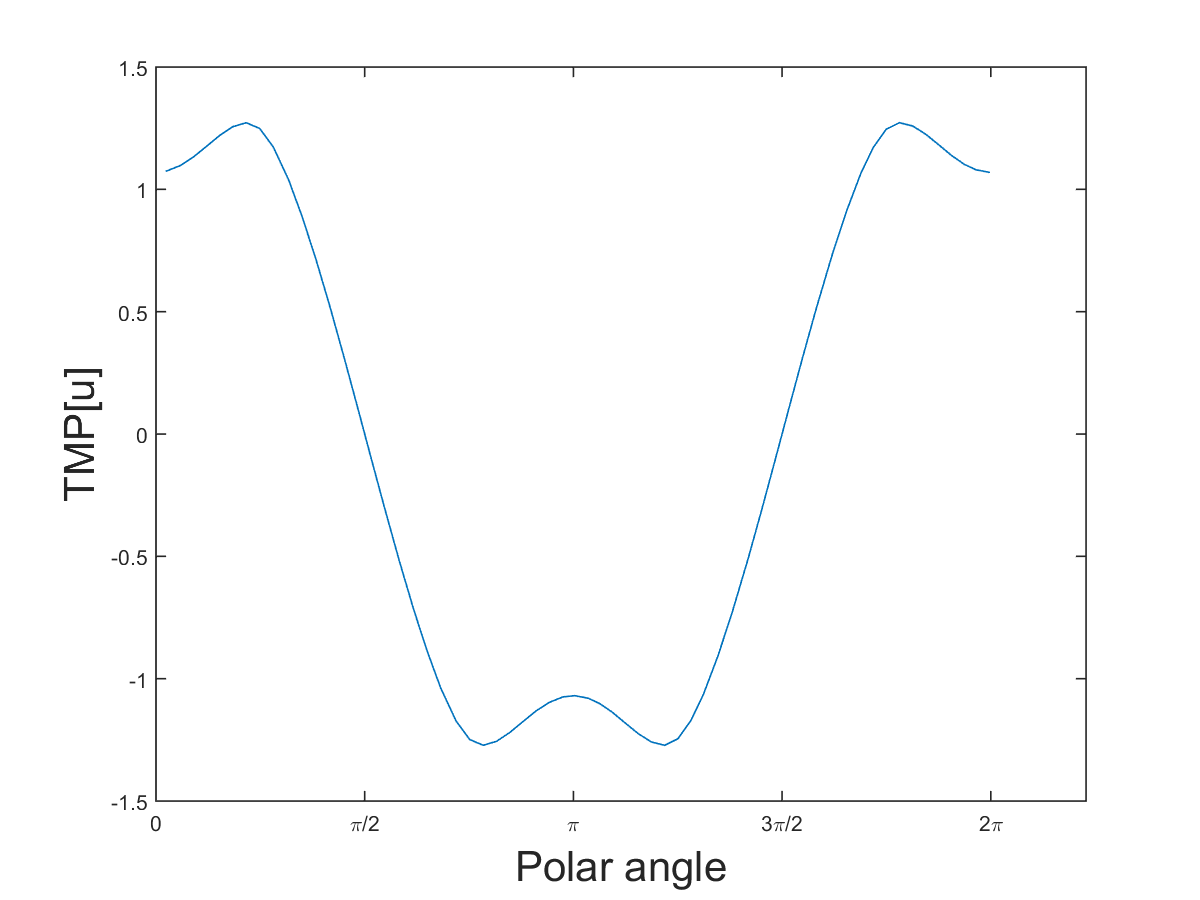}}
\caption{\textbf{(a)}~Evolution of the transmembrane potential (TMP) $v$ at the pole of the cell. \textbf{(b)}~TMP along the cell membrane after 2 $\mu s'$. }
\label{fig:CellModel}
\end{figure}

\subsection{Homogenization for electropermeabilization model}
In this section, we show the sensitivity of the effective parameters $\sigma_0$, $A^0$, and $A^1$ in \eqref{eq:FormLimit} to
\begin{itemize}
\item the conductivities $\sigma_o$ and $\sigma_i$;
\item the shape of the cell with membrane $\Gamma$;
\item the volume fraction $f=\frac{\vol (Y_i)}{\vol(Y)}$;
\item the lattice of the cells in the domain $\Omega$. \end{itemize}
We perform four experiments, the results of which are found in Table \ref{tab:Sensitivity}.
\medskip

\begin{figure}
\centering
\includegraphics[width=10cm]{\PicPath 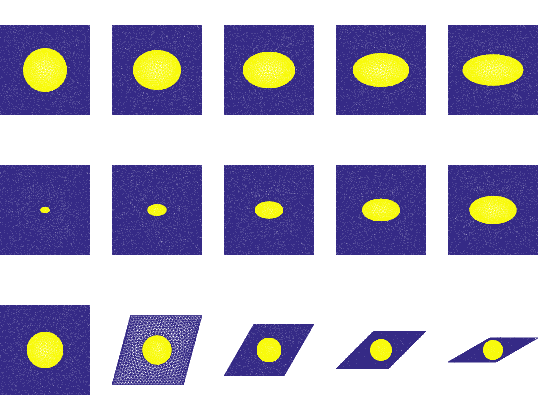}
\caption{Cell shapes used in numerical examples (see text and Table~\ref{tab:Sensitivity}). Example~1 uses the first mesh. Example~2 uses the cells in the first row. Example~3 uses the cells in the second row. Example 4 uses the cells in the last row.}
\label{fig:Meshes}
\end{figure}

\textbf{Example 1.} We fix the shape and size of the cell and change the ratio of the interior and exterior conductivities $\sigma_i$ and $\sigma_e$.

\textbf{Example 2.} In this example, we show how the shape of the cell membrane produces different effective anisotropy properties. We fix conductivities and the volume fraction of the cell, but take as cell shapes ellipses with different excentricity $a/b$.

\textbf{Example 3.} 
We investigate the effect of different volume fractions of a cell with the same shape. 

\textbf{Example 4.} 
In this example, we show how the angle of the lattice in which the cells are arranged affects the effective parameters.

\medskip

For all these experiments, Table \ref{tab:Sensitivity} presents the reactions of the effective conductivity $\sigma_0$ and the effective anisotropy properties $A^0$ and $A^1(0)$ to the microscopical change. One sees clearly that  $\sigma_0$, as well as $A^0$ and $A^1$ react to a change of cell and conductivity parameters. Most of the sensitivity functions are in fact monotonic.

The best contrast is seen in:
\begin{itemize}
\item the reaction of $\sigma_0$ to the change in conductivity $\sigma_i/\sigma_e$ and to a change in the lattice angle $\phi$;
\item the reaction of both $A^0$ and $A^1$ to the cell shape.
\end{itemize}
The volume fraction alone does not show so much contrast in the anisotropy of the effective parameters. 

\medskip

Given the results of the sensitivity analysis,  it is promising to infer shape parameters  from macroscopic effective properties in electropermeabilization, as it was done in \cite{laure} from multifrequency admittivity measurements.

\begin{table}
\begin{tabular}{ccc}
effective conductivity $\sigma_0$ & eigenvalues $\lambda_1/\lambda_2$ of $A^0$ & eigenvalues $\lambda_1/\lambda_2$ of $A^1(0)$. 
\vspace{0.05cm}\\
\hline
\hline

\\
\multicolumn{3}{l}{Example 1: Difference in conductivity (ratio $\sigma_i/\sigma_e$ of interior and exterior conductivity).} \\
 \includegraphics[width=5cm]{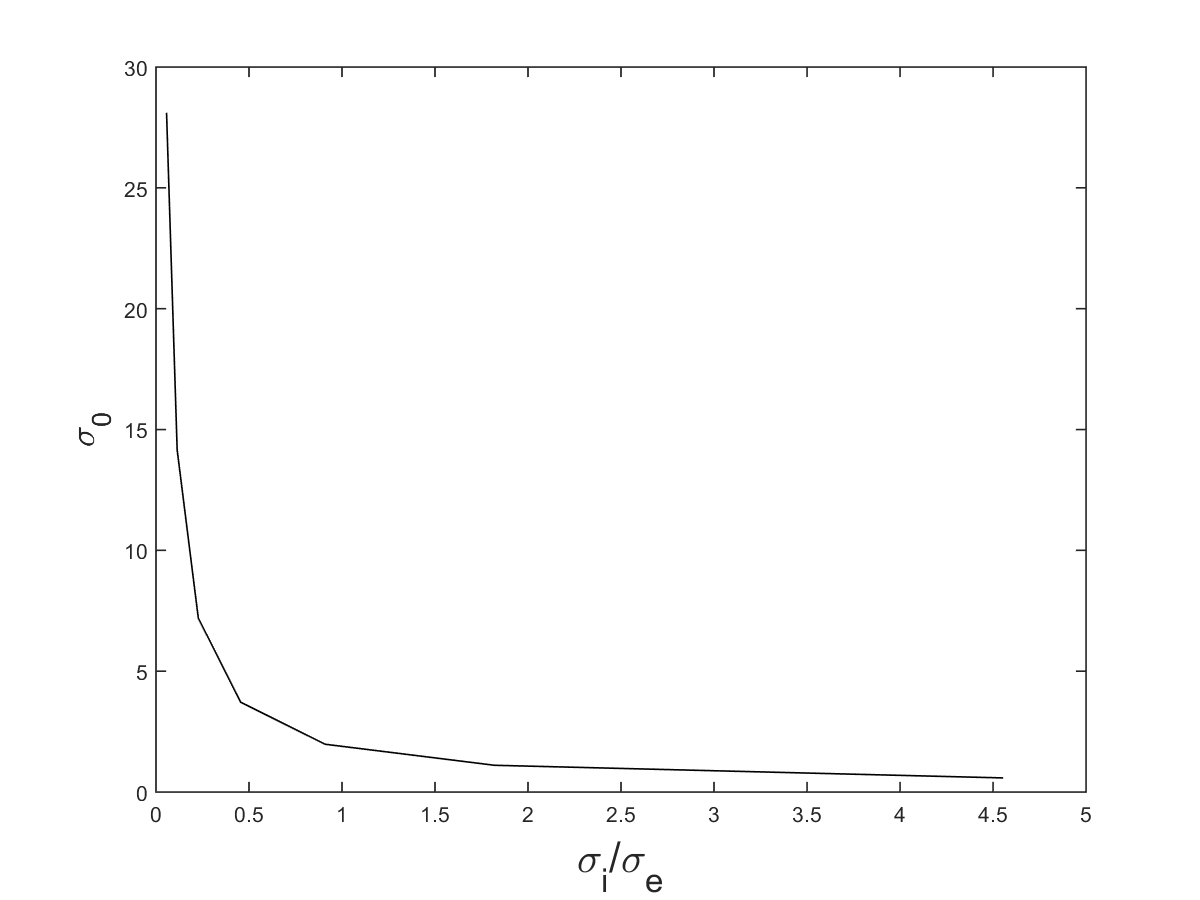}
& \includegraphics[width=5cm]{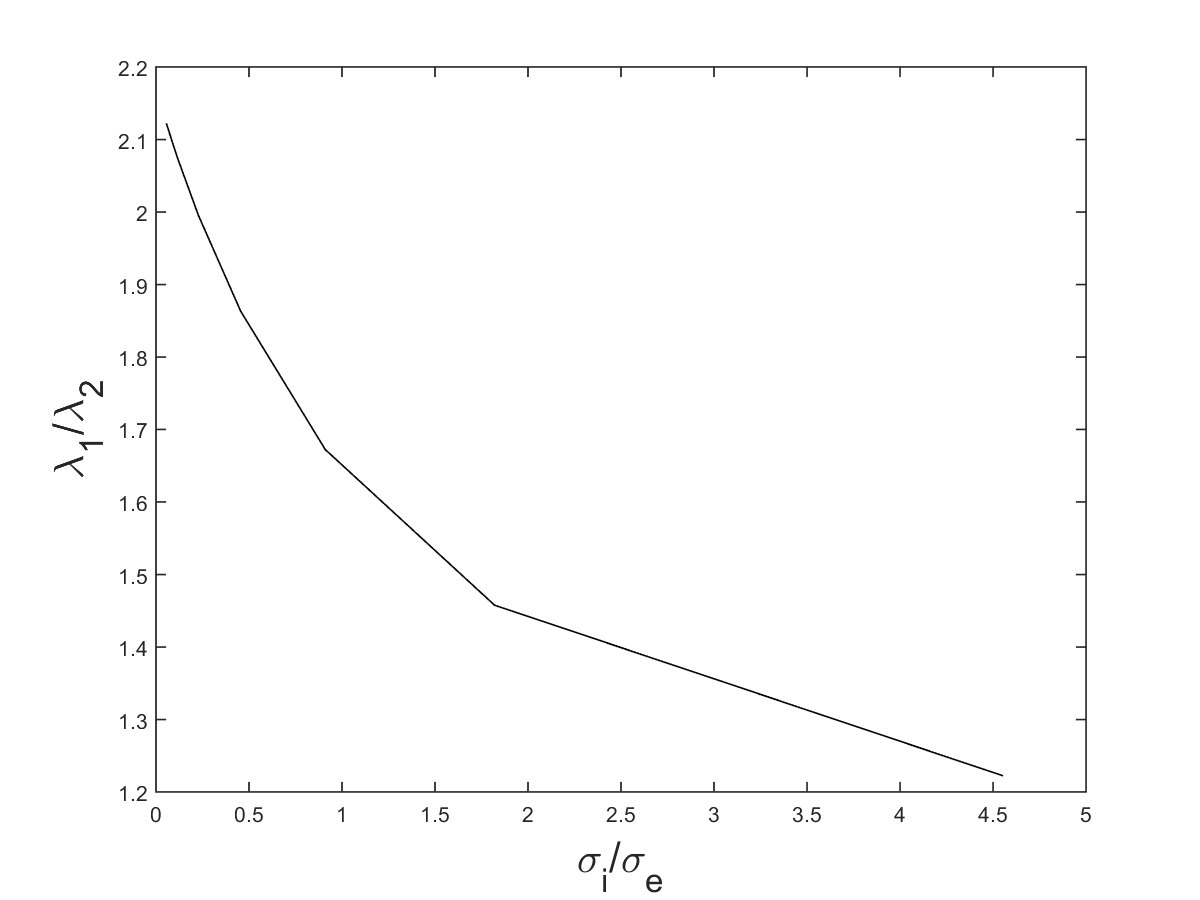}
& \includegraphics[width=5cm]{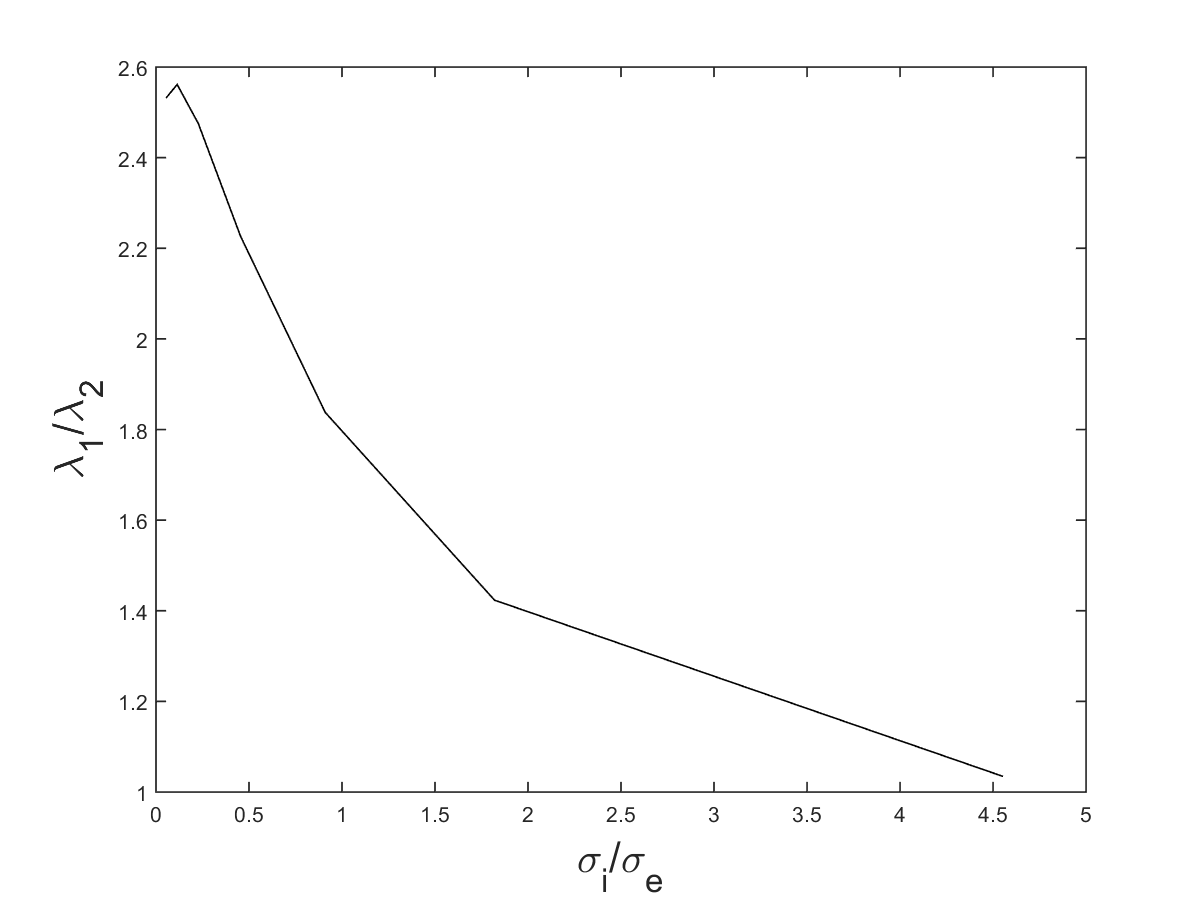}
\\

\hline

 \\

\multicolumn{3}{l}{Example 2: Difference in cell shape: change of the excentricity $a/b$ (see Fig.~\ref{fig:Meshes}, 1st row).} \\
 \includegraphics[width=5cm]{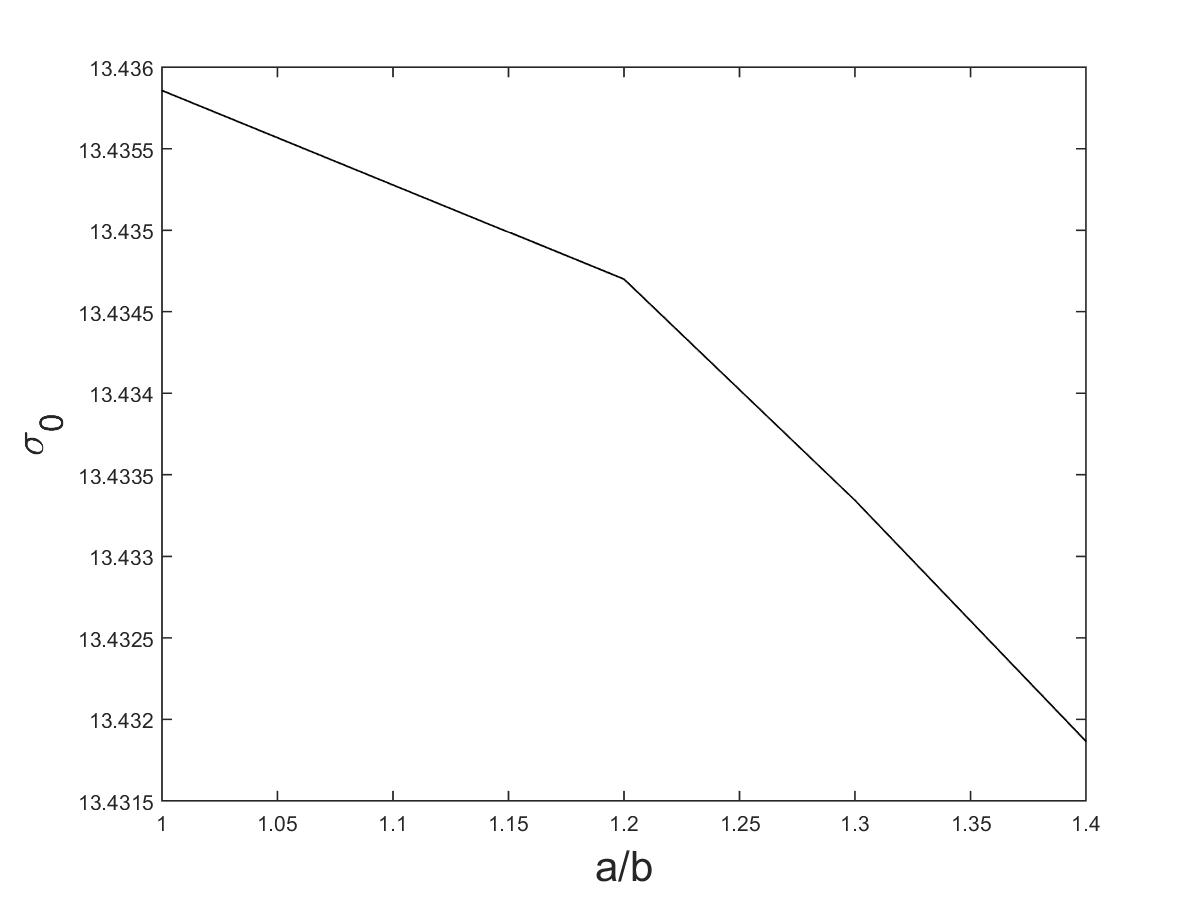}
& \includegraphics[width=5cm]{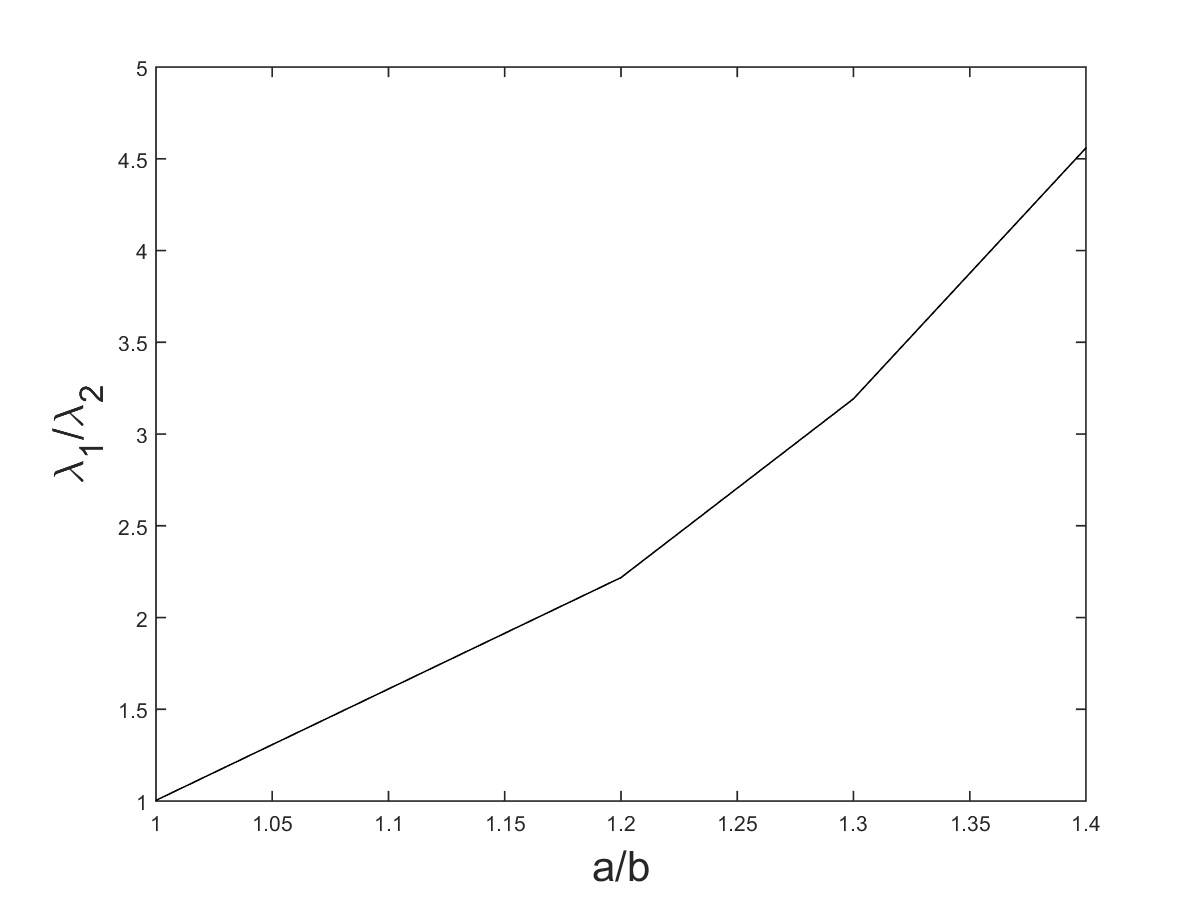}
& \includegraphics[width=5cm]{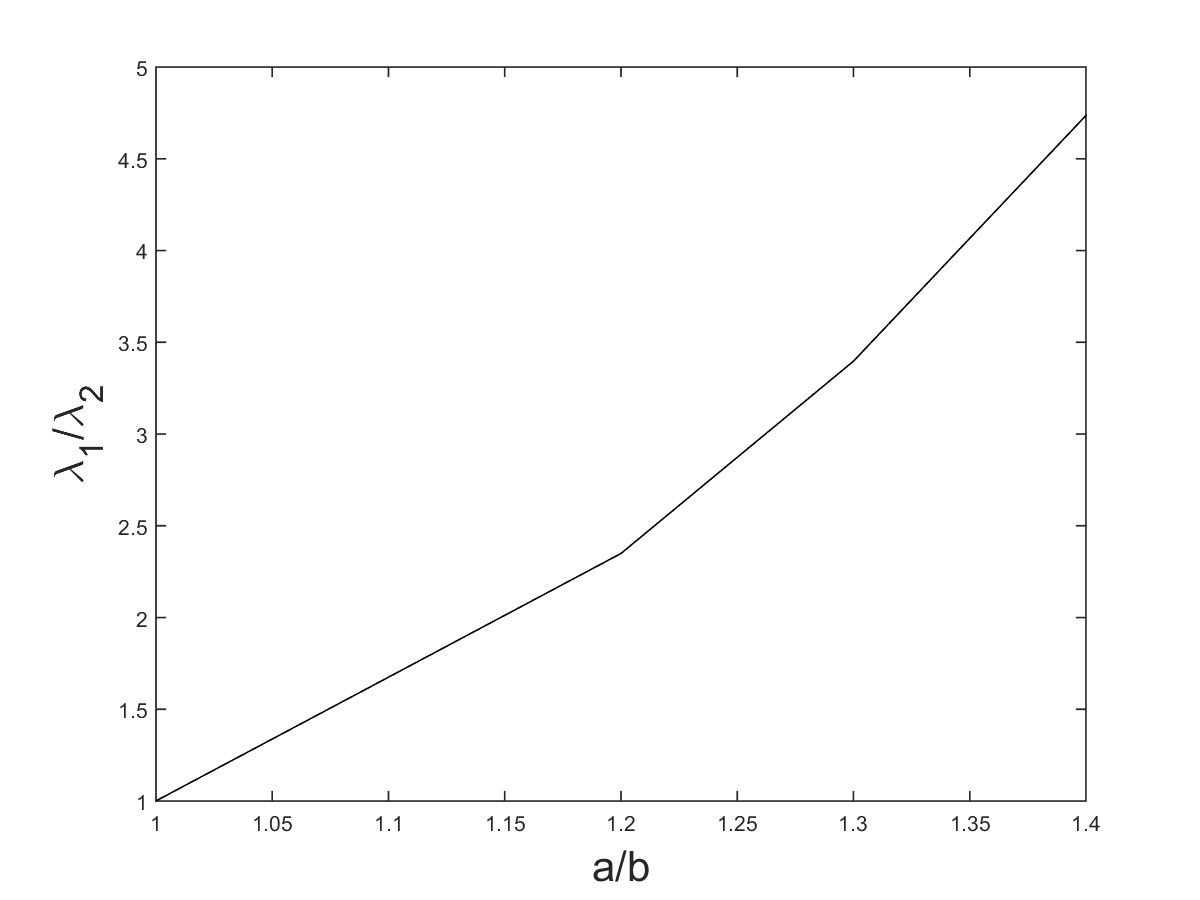}
\\

\hline

\\

\multicolumn{3}{l}{Example 3: Difference in volume fraction of the cells (see Fig.~\ref{fig:Meshes}, 2nd row).} \\
 \includegraphics[width=5cm]{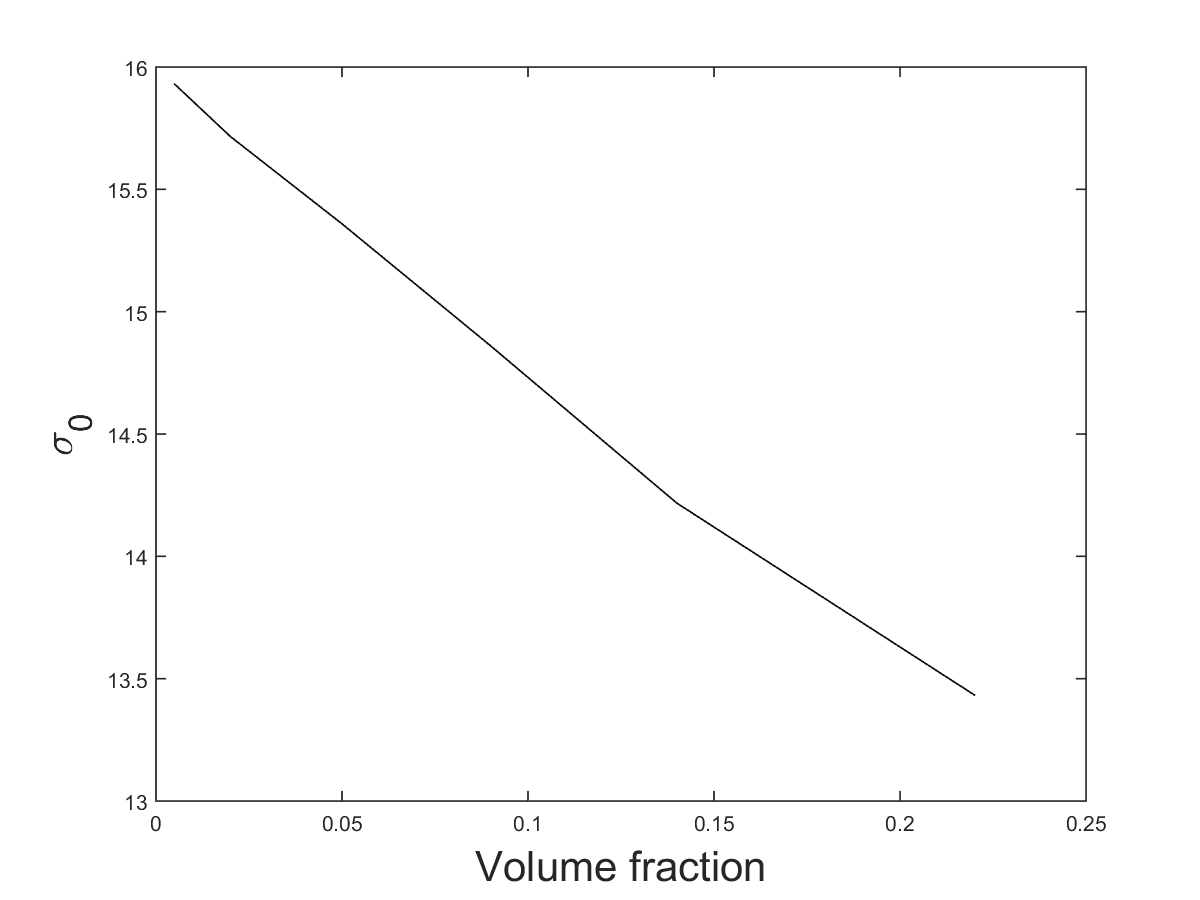}
& \includegraphics[width=5cm]{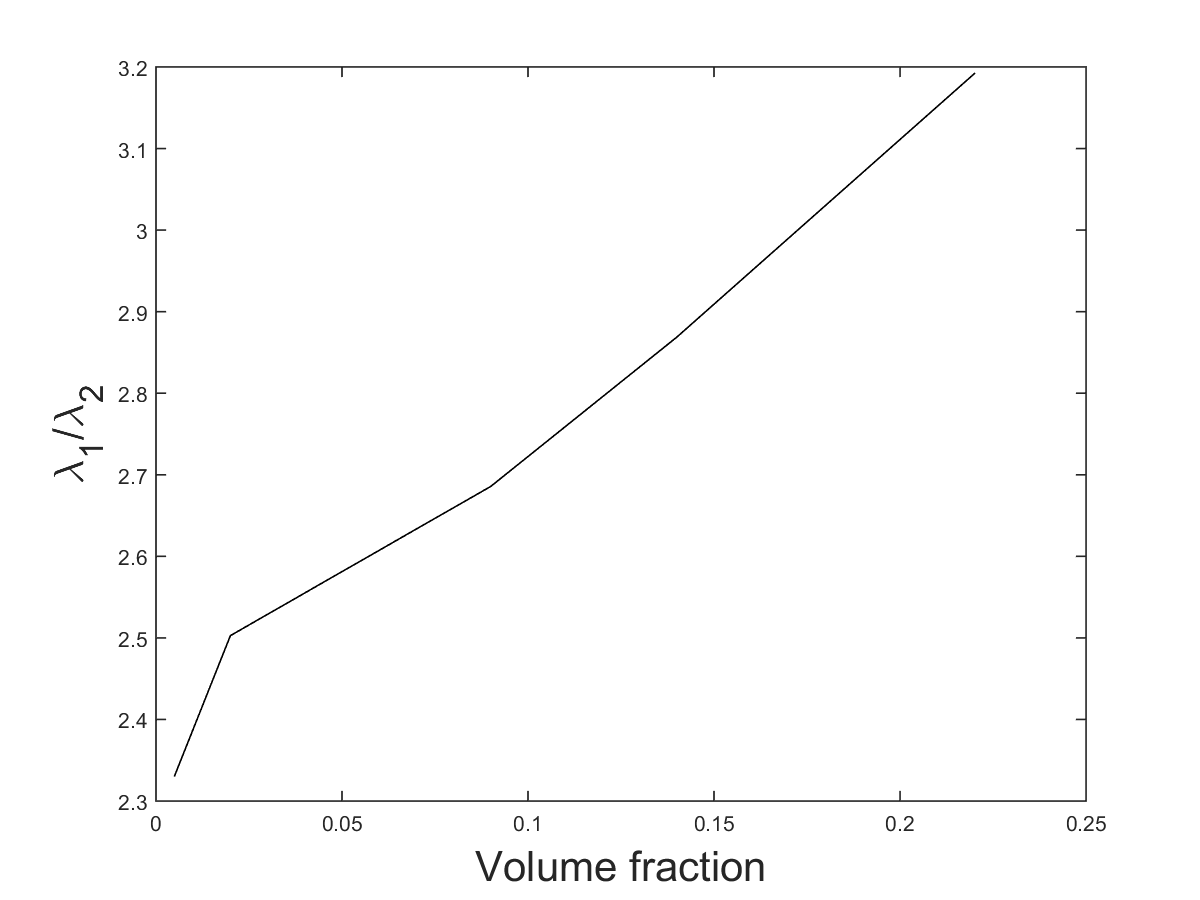}
& \includegraphics[width=5cm]{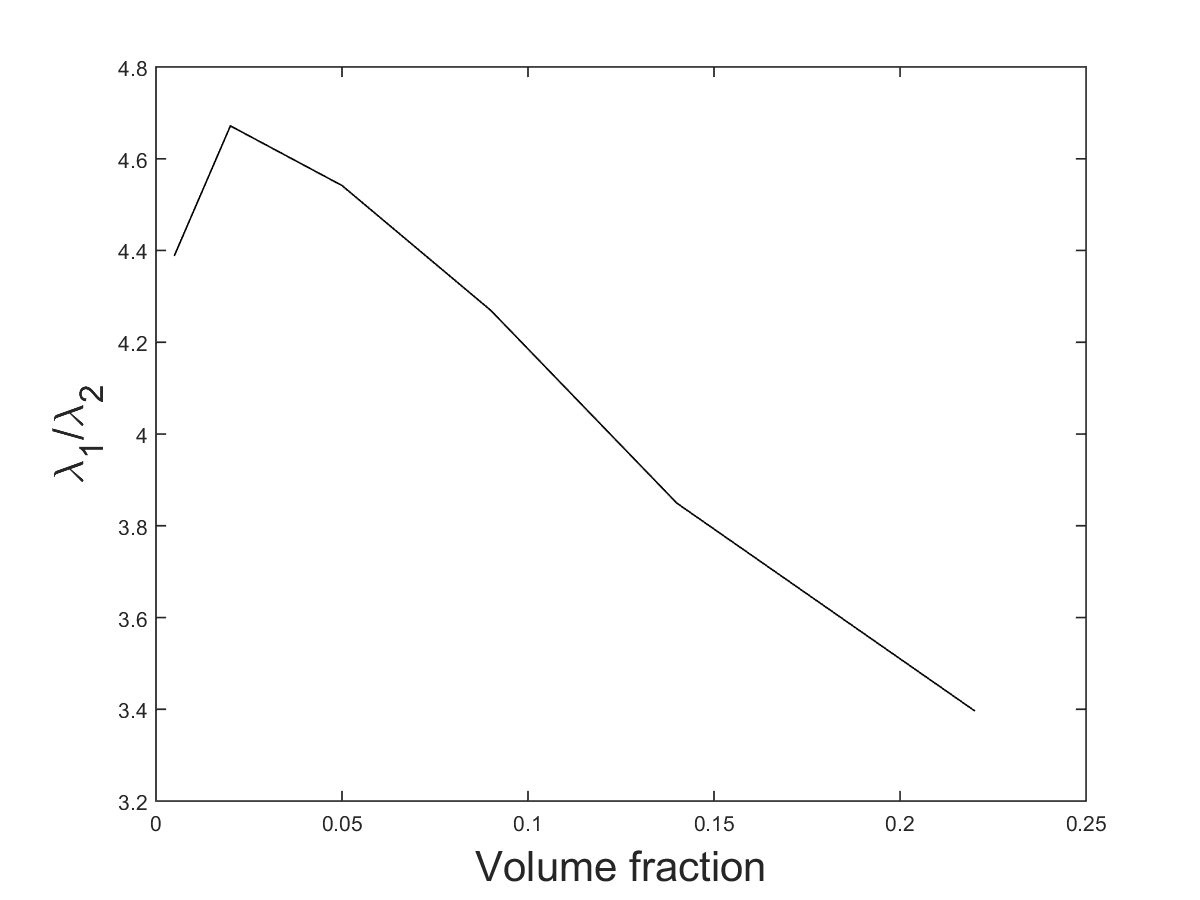}
\\

\hline

\vspace{0.05cm} \\

\multicolumn{3}{l}{Example 4: Difference in angle $\phi$ of the lattice arrangement (see Fig.~\ref{fig:Meshes}, 3rd row).} \\
 \includegraphics[width=5cm]{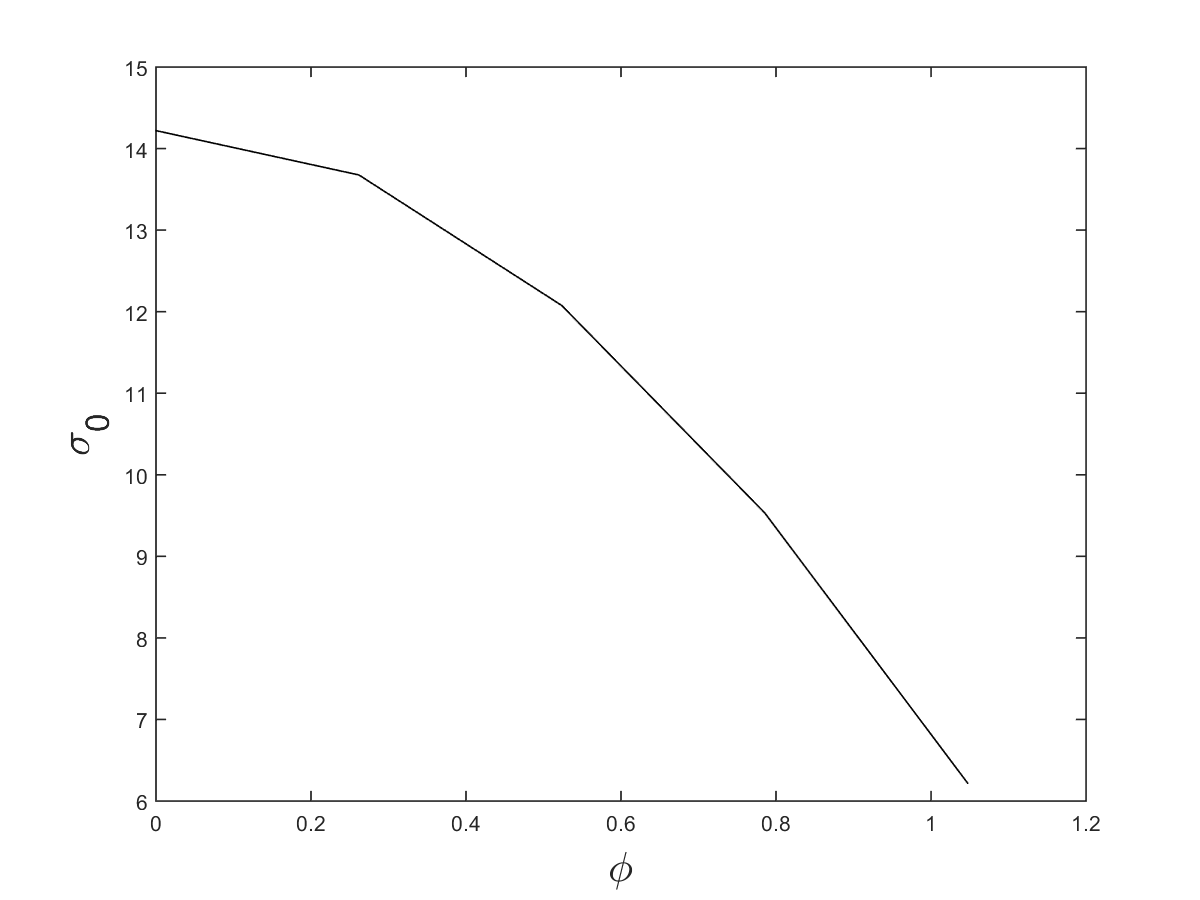}
& \includegraphics[width=5cm]{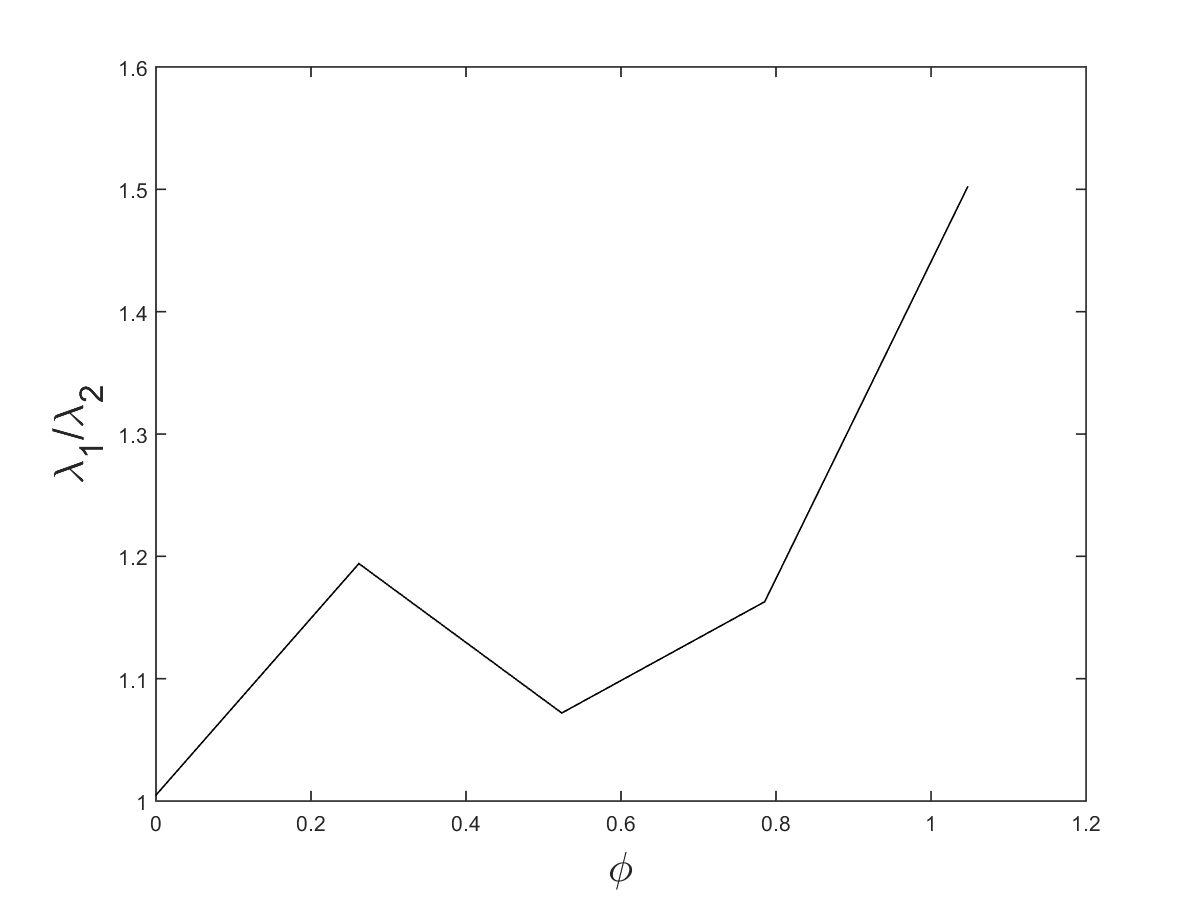}
& \includegraphics[width=5cm]{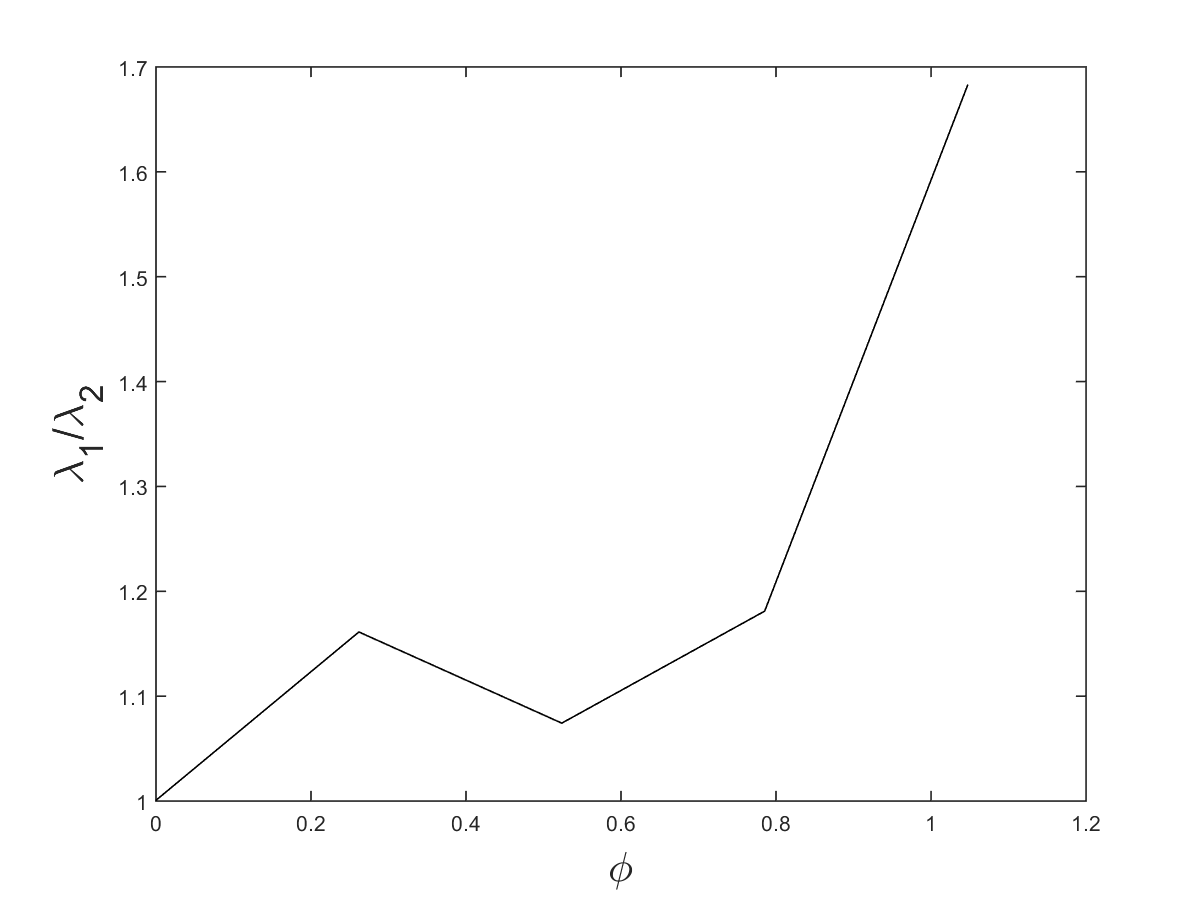}

\end{tabular}
\caption{Changes in microscopic parameters and the reaction of the effective parameters in \eqref{eq:FormLimit}.}
\label{tab:Sensitivity}
\end{table}

\section{Concluding remarks}
\label{sec:Discussion}

We introduced a homogenization scheme relating critical microscopic and macroscopic
quantities in electropermeabilization. The sensitivity analysis of the effective
parameters showed this dependence and opens the door to solve the inverse problem to
monitor those critical microscopic quantities in practice.

While setup optimization for electropermeabilization has been studied using computer simulations, for instance,  in
\cite{MikSnoZupKosCer10,SugYosMorWatHas01,GolRub12,MikSemMekMir00,MikBerSemCemDem98},
from our approach comes an additional constraint: for mapping of the effective
parameters $A^1$ and $A^0$, two currents have to be applied which are nowhere
parallel. An electrode configuration providing this allows for unique reconstruction~
\cite{KimKwoSeoWoo03}.

\appendix
\section{Convergence for homogenization}
We give here the outline of the method used in \cite{AmaAndBisGia04}. It shows how Lemma \ref{lem:K3} is used to prove Theorem \ref{thm:Convergence} for our application.

\label{app:ConvTheo}
\begin{theorem} For the solution $u_\varepsilon$ in~\eqref{eq:PeriodicSystem} and the homogenized solution $u_0$ in~\eqref{eq:FormLimit}, we have the convergence \[
	u_\varepsilon\to u_0
\]
weakly in $L^2([0,T]\times\Omega)$ and strongly in $L^1_{\mathrm{loc}}([0,T],\Omega)$.
\end{theorem}
\begin{proof}
From the estimate (\ref{prop:Energy}) we get, extracting subsequences if needed
\beal{eq:Limits}
	&u_\varepsilon\to u_0,\qquad\sigma\nabla u_\varepsilon\to\xi\qquad &\text{ weakly in } L^2([0,T]\times\Omega),\\
	&u_\varepsilon\to u_0\qquad &\text{ strongly in } L^1_{\mathrm{loc}}([0,T],\Omega).
\eeal
Next, consider the weak formulation of system \eqref{eq:PeriodicSystem}:
\beql{eq:WeakFormulation}
\begin{split}
	\int_0^T\int_\Omega\sigma\nabla u_\varepsilon\cdot\nabla\psi\; dx\;dt
	+ \frac{1}{\varepsilon}\int_0^T\int_{\Gamma_\varepsilon} \sigma_m([u_\varepsilon]_M)[u_\varepsilon]_M[\psi]\;dS\;dt \\
	-\frac{c_m}{\delta}\int_0^T\int_{\Gamma_\varepsilon} [u_\varepsilon]\frac{\partial}{\partial t}[\psi]\;dS\;dt
	- \frac{c_m}{\delta} \int_{\Gamma\varepsilon} [u_\varepsilon](0)[\psi](0)\;dS = 0.
	\end{split}
\eeq
The general idea is to pass to the limit $\varepsilon\to 0$ in this equation, and therefore to obtain the equation for $u_0$. This is possible for special test functions $\psi$.

 Choose for $\psi$ the functions $\varphi\;w_h^\varepsilon(x,t)$ for $h=1,\ldots, d$, where $\varphi$ is a smooth with compact support on $\Omega$, and $w_h^\varepsilon$ is built by the cell functions $\chi^1$ and $\chi^2$:
\[
	w_h^\varepsilon(x,t):=x_h-\varepsilon\chi_h^0\left(\frac{x}{\varepsilon}-\varepsilon\right)\int_t^T\chi_h^1\left(\frac{x}{\varepsilon},\tau-t\right)\; d\tau.
\]
For this definition, given in \cite[(5.1)]{AmaAndBisGia04} one has the weak formulation in \cite[(5.2)-(5.4)]{AmaAndBisGia04}.

 By subtracting the weak equation \eqref{eq:WeakFormulation} for $\psi=w_h^\varepsilon(x,t)$ and the equations  \cite[(5.2)-(5.4)]{AmaAndBisGia04}, one can isolate the term $\int\int\sigma\nabla u_\varepsilon\nabla\varphi w_h^\varepsilon\;dx\;dt $:
\beql{eq:IsolationK}
	\int_0^T\int_{\Gamma_\varepsilon}\sigma\nabla u_\varepsilon\nabla\varphi w_h^\varepsilon\;dx\;dt = K_{1\varepsilon} + K_{2\varepsilon} + K_{3\varepsilon},
\eeq
with
\beal{eq:DefinitionK}
K_{1\varepsilon} &= \int_0^T\int_{\Gamma_\varepsilon}\sigma\nabla w_h^\varepsilon \nabla\varphi u_\varepsilon\;dx\;dt, \\
K_{2\varepsilon} &= -c_m \varepsilon \int_{\Gamma_\varepsilon} (S_1(x,\frac{x}{\varepsilon}) +R_{\varepsilon})\varphi \int_0^T[\chi_h^2](\frac{x}{\varepsilon},\tau) d\tau dS,\\
K_{3\varepsilon} &=  \frac{1}{\varepsilon}\int_0^T\int_{\Gamma_\varepsilon} \Bigl(\sigma_m(0,t)[u_\varepsilon]-\sigma_m([u_\varepsilon],t)[u_\varepsilon]\Bigr)[w_h^\varepsilon] \varphi\;dS\; 
dt.
\eeal
 The limits of $K_{1\varepsilon}$ and $K_{2\varepsilon}$ are the same as in \cite[p.18]{AmaAndBisGia04}, whereas for the limit $K_{3\varepsilon}$, one can show that $K_{3\varepsilon}\to 0$ by  Lemma \ref{lem:K3}.
 One can take then the limit $\varepsilon\to 0$ in \eqref{eq:IsolationK} in order to obtain information on the specific form of the limit $u_0$ in \eqref{eq:Limits}. We get
\beql{eq:SchwereGeburt}
\begin{split}
	-\int_0^T\int_\Omega\xi\cdot\nabla\varphi x_h\;dx\;dt = \int_0^T\int_\Omega \varphi(x)F_h(x,\tau);dx\;d\tau
	\\ + \int_0^T\int_\Omega u_0(x,t)(\sigma_0 I + A^0)\b e_h + \int_0^t u_0(x,\tau)A^1(t-\tau)\b e_h\;d\tau \cdot\nabla\varphi(x)\;dx\;dt
	\end{split}
\eeq
with $A^0$, $A^1$, $\b F$ defined as in \eqref{eq:DefA_F}.
 Choosing $\psi=\varphi\;x_h$ in \eqref{eq:WeakFormulation}, combining with \eqref{eq:SchwereGeburt}, and differentiating in $T$ gives then expressions which show that $u_0\in L^2([0,T],H^1(\Omega))$ and that actually \eqref{eq:FormLimit} is the correct equation of the limit $u_0$.

\end{proof}

%
%\bibliographystyle{\BibPath plain_talks}
%\bibliography{\BibPath strings,\BibPath articles,\BibPath books,\BibPath infmath,\BibPath infmath_books,\BibPath infmath_report,\BibPath infmath_talks,\BibPath infmath_theses,\BibPath inproceedings,\BibPath preprints,\BibPath proceedings,\BibPath theses,\BibPath unsubmitted,\BibPath websites,\BibPath x_ENS}
\small
\def\cprime{$'$}
  \providecommand{\noopsort}[1]{}\def\ocirc#1{\ifmmode\setbox0=\hbox{$#1$}\dimen0=\ht0
  \advance\dimen0 by1pt\rlap{\hbox to\wd0{\hss\raise\dimen0
  \hbox{\hskip.2em$\scriptscriptstyle\circ$}\hss}}#1\else {\accent"17 #1}\fi}

\end{document}